\documentclass[a4paper,12pt]{article}
\usepackage[cp1251]{inputenc}
\usepackage[russian,english]{babel}
\usepackage{amsmath, amsthm, amsfonts, amssymb}

\makeatletter
\@addtoreset{equation}{section}
\makeatother
\makeatletter
\@addtoreset{section}{part}
\makeatother
\makeatletter
\@addtoreset{thm}{section}
\makeatother
\makeatletter
\@addtoreset{lem}{section}
\makeatother
\makeatletter
\@addtoreset{cor}{section}
\makeatother
\makeatletter
\@addtoreset{expl}{section}
\makeatother
\makeatletter
\@addtoreset{remk}{section}
\makeatother
\makeatletter
\@addtoreset{nsl}{section}
\makeatother
\makeatletter
\@addtoreset{defn}{section}
\makeatother

\newcommand{\mbQ}{{\mathbb Q}}

\newcommand{\cT}{{\mathcal T}}
\newcommand{\cM}{{\mathcal M}}
\newcommand{\cP}{{\mathcal P}}
\newcommand{\cB}{{\mathcal B}}

\newcommand{\cE}{{\mathcal E}}
\newcommand{\cF}{{\mathcal F}}

\newcommand{\ve}{\varepsilon}
\newcommand{\mbR}{{\mathbb R}}

\theoremstyle{plain}
\newtheorem{thm}{Theorem}[section]
\newtheorem{lem}{Lemma}[section]

\theoremstyle{definition}
\newtheorem{defn}{Definition}
[section]
\newtheorem{expl}{Example}
[section]
\theoremstyle{remark}
\newtheorem{remk}{Remark}[section]
\theoremstyle{Corollary}

\theoremstyle{Proposition}

\begin{document}
\begin{center}
{\bf \Large{Self-intersection local times of random fields in stochastic flows}}
\end{center}
\large
\begin{center}
Andrey Dorogovtsev, Alexander Gnedin, Olga Izyumtseva
\end{center}
{\bf Abstract.}
In this article we study transformations of Gaussian field by stochastic flow on the plane. A stochastic flow is a solution to the equation with interaction whose coefficients depend on the occupation measure of the field. We consider nonsmooth Gaussian field, which has self-intersection local times of any multiplicity. In the article we prove the existence of self-intersection local times for the transformed field and study its asymptotics.

\section{Introduction}
\label{Introduction}
In present article we study geometric characteristics of nonsmooth Gaussian random field in the flow of interacting particles on the plane. Our model has two essential features. Firstly, our field is nonsmooth and its geometric characteristics are  "numbers of self-intersections." Secondly, the motion of particles depends on the field itself. For description of particle motion in the random media we use equations with interaction invented by A.A. Dorogovtsev  \cite{16} (p. 54), \cite{9}.
Let us describe these equations more precisely.
\begin{defn}
\label{defn1.1}\cite{16}
The following stochastic differential equation
\begin{equation}
\label{eq1.1}
\begin{cases}
dx(u,t)=a(x(u,t),\mu_t)dt+\int_{\mbR^2}b(x(u,t),\mu_t,p)W(dp,dt)\\	
x(u,0)=u,\ u\in\mbR^2\\
\mu_t=\mu_0\circ x(\cdot,t)^{-1}
\end{cases}
\end{equation}
is said to be an equation with interaction.
\end{defn}
Probability measure $\mu_0$ on $\mbR^2$ is an initial distribution of mass of particles moving in random media. Coefficients depend on spatial variable and measure which describes the mass of particles in space. More precisely coefficient $a$ describes the speed of particle obtained under the action of external forces. Coefficient $b$ and Brownian sheet $W$ on $\mbR^2\times [0;\infty)$ describe small random perturbations acting on our particle. A trajectory of particle with starting point $u\in\mbR^2$ is described by $\{x(u,t),\ t\geq0\}.$ The measure $\mu_t$ is the distribution of the mass of particles at time $t.$ It was proved in \cite{16} (p. 55) that if coefficients of the equation \eqref{eq1.1} satisfy Lipschitz condition with respect to spatial and measure-valued variables,
then there exists the unique solution to \eqref{eq1.1}. Moreover, if coefficients are two times continuously differentiable with respect to spatial variable with bounded derivatives, then for every fixed $t\geq 0$ the solution $x(u,t),\ u\in\mbR^2$ is diffeomorphism, almost surely \cite{16} (p. 66).

Describe how a random field can be included in our stochastic flow. Let $\eta(u),\ u\in D\subset\mbR^2$ be a $\mbR^2$-valued random field.
\begin{defn}
\label{defn1.1}
A probability measure $\mu$ defined on Borel subsets of $\mbR^2$ as follows
$$
\mu(A)=\int_D1_{A}(\eta(u))du,\ A\in\cB(\mbR^2)
$$
is said to be an occupation measure of a field $\eta.$
\end{defn}
It is not difficult to check that the following statement holds.
\begin{lem}
\label{lem1.1}
For any bounded and measurable function $\varphi:\mbR^{2k}\to\mbR$ the following relation holds
$$
\int_{\mbR^{2k}}\varphi(v_1,\ldots,v_k)\mu(dv_1)\ldots\mu(dv_k)=
$$
\begin{equation}
\label{eq1.3}
=\int_D\ldots\int_D\varphi(\eta(u_1),\ldots,\eta(u_k))du_1\ldots du_k.
\end{equation}
\end{lem}
Formally applying \eqref{eq1.3} for the function
$$
\varphi(v_1,\ldots,v_k)=\prod^{k-1}_{i=1}\delta_0(v_{i+1}-v_i),\ v_1,\ldots,v_k\in\mbR^2
$$
one can obtain the following relation
$$
\int_{\mbR^{2k}}\prod^{k-1}_{i=1}\delta_0(v_{i+1}-v_i)\mu(dv_1)\ldots\mu(dv_k)=
$$
$$
=\int_D\ldots\int_D\prod^{k-1}_{i=1}\delta_0(\eta(u_{i+1})-\eta(u_i))du_1\ldots du_k.
$$
For nonsmooth random fields $\eta$ formal expression
$$
\int_D\ldots\int_D\prod^{k-1}_{i=1}\delta_0(\eta(u_{i+1})-\eta(u_i))du_1\ldots du_k
$$
is said to be a self-intersection local time. It can be considered as geometric characteristics of field $\eta$ \cite{31}.

Consider the equation with interaction \eqref{eq1.1}. Let $\mu_0$ be an occupation measure of the field $\eta.$ Then one can see that $\mu_t$
is an occupation measure of a new field $x(\eta(u),t),\ u\in D.$
Important that the motion of separate points of a field depends on a shape of a whole field. Coefficients of equation which contain an occupation measure describe this dependence.

There exists a large amount of literature related to the randomly moving curves and manifolds. Here we mention just a few of such papers in order to emphasize
the new features of our model. Let us start from T. Funaki paper \cite{10}.
T. Funaki proposed to consider a string in $\mbR^d$ as a continuous mapping from interval $[0;1]$ to $\mbR^d.$ The evolution of the string is described by the function $x_t(\sigma),\ \sigma\in[0;1],\ t>0$ such that for any $t>0$ $x_t\in C([0;1],\mbR^d).$ T. Funaki introduced the equation of string motion
\begin{equation}
\label{eq2.3}
\frac{\partial x_t(\sigma)}{\partial t}=B(x_t(\sigma))\xi+a(x_t(\sigma))+\frac{k}{2}\frac{\partial^2x_t(\sigma)}{\partial \sigma^2},
\end{equation}
where $\xi$ is $\mbR^d$-valued space-time white noise defined on $[0;1]\times[0;+\infty).$
The first term of \eqref{eq2.3} describes the interaction of the string with the random media, the second term describes the action of external forces and the third term describes the interaction of neighboring parts of the string.
Coefficients $a:\mbR^d\to\mbR^d,\ B:\mbR^d\to\mbR^{d\times d}$
are Lipschitz functions, $k>0$ is a parameter of elasticity of string.
Also it is assumed that $B$ is bounded. The author defines mild solutions to equation \eqref{eq2.3} \cite {15} with respect to different boundary conditions and the initial state $x_0\in C([0;1],\mbR^d).$ Namely, equation \eqref{eq2.3} is substituted by the following integral equation
$$
x_t(\sigma)=\int^1_0p_t(\sigma,\tau)x_0(\tau)d\tau+\int^t_0\int^1_0p(t-s,\sigma,\tau)B(x_s(\tau))W(ds,d\tau)+
$$
$$
+\int^t_0\int^1_0p(t-s,\sigma,\tau)a(x_s(\tau))dsd\tau,
$$
where $p(t,\sigma,\tau)$ is fundamental solution to $\frac{k}{2}\frac{\partial^2}{\partial \sigma^2}$ with the corresponding boundary conditions. Note that solution to \eqref{eq2.3} depends on $k$. In the paper author studies asymptotic property of solution for large $k.$ Asymptotics depends on initial boundary conditions. If both ends of string move freely, then string shrinks to a single point. If one end is fixed, then string contracts to the fixed point. In the case of two fixed ends the solution converges to the interval in $\mbR^d$ defined by fixed points. Moreover the equation of Brownian string in a potential with mentioned boundaries conditions is considered. In this case author defines the stationary measure and investigates its properties. In dimension 2 recurrent properties of solution is discussed.

M.Kardar, G.Parisi and Y.-C.Zhang in \cite{11} proposed a model for a time evolution of a profile of growing interface. An interface profile is described by a height $h(u,t),\ u\in\mbR^d, t>0.$ The Langevin equation for a local growth of a profile is given by
\begin{equation}
\label{eq1.4}
\frac{\partial h}{\partial t}=\nu\Delta h+\frac{\lambda}{2}(\nabla h)^2+\xi,
\end{equation}
where $\xi$ is $\mbR^d$-valued space-time white noise defined on $[0;1]\times[0;+\infty)$ and $\Delta$ is Laplacian. The first term of the right-hand side describes relaxation of an interface by a surface tension $\nu.$ The second term is a lowest-order nonlinear term that can appear in an interface growth equation. The solution to equation \eqref{eq1.4} authors define as mild solution \cite{15} as it was done in T. Funaki paper.

It must be mentioned that proposed evolutions do not take into account the changing of geometric characteristics of evolving curve. Attempt to describe the evolution of smooth curve taking into account the changing of its geometric characteristics was done by Landau and Lifshitz \cite{13}. They considered the space of knots, i.e. the space of all smooth mappings $\theta: S^1\to\mbR^3$ such that for any $y\in S^1:\ \theta^{\prime}(y)\neq0$ and $\theta$ has no double points. Moreover any two maps with the same images are equal. Here
$$
S^1=\{u\in\mbR^2:\ \|u\|=1\}.
$$
The time evolution of knot $\theta(x,t)$ of the curve $\theta(x,0),\ x\in S^1$ is described by filament equation
$$
\frac{\partial\theta}{\partial t}=k(x,t)\frac{\partial\theta}{\partial x}\times \frac{\partial^2\theta}{\partial x^2},
$$
where $k(x,t)$ is the curvature of the curve at the point $x$ and time $t.$
It occurs that filament equation is closely related to Helmholtz approximation to Euler equation \cite{13}. Helmholtz proposed to consider instead of liquid motion the motion of finite number of infinitesimal vortices. If positions of vortices are $r_1,\ldots,r_n,$ then Helmholtz equation has the representation
$$
\frac{dr_i(t)}{dt}=-\frac{1}{\pi}\sum_{i\neq j}k_j\nabla\ln\|r_i(t)-r_j(t)\|.
$$
Note that this equation can be considered as an equation with interaction with the initial distribution
$$
\mu_0=\sum^{n}_{i=1}k_i\delta_{r_i(0)}.
$$
Really, if one put $b=0$ and
$$
a(v,\mu)=-\frac{1}{\pi}\int_{\mbR^2}1_{v\neq x}\nabla\ln\|v-x\|\mu(dx)
$$
in \eqref{eq1.1}, then one can obtain Helmholtz equation.
Suppose that measures
$$
\sum^{n}_{i=1}k_i\delta_{r_i(0)},\ n\geq1
$$
approximate an occupation measure of smooth curve. By considering only local terms in Helmholtz equation, i.e.
$$
\frac{dr_i(t)}{dt}=-\frac{1}{\pi}(k_{i-1}\nabla\ln\|r_i(t)-r_{i-1}(t)\|+k_{i+1}\nabla\ln\|r_i(t)-r_{i+1}(t)\|
$$
and passing to the limit when $n\to+\infty$
one can obtain the filament equation. If one consider not only local terms in Helmholtz equation, then after passing to the limit one can obtain the equation with interaction with the initial measure equals an occupation measure of the smooth curve, coefficients $b=0$ and
$$
a(v,\mu)=-\frac{1}{\pi}\int_{\mbR^2}1_{v\neq x}\nabla\ln\|v-x\|\mu(dx).
$$

In the paper \cite{12} authors made an attempt to trace the changing of geometric characteristics of surface obtained as an image of some manifold under action of an isotropic stochastic flow. In \cite{12} Y. LeJan and M. Cranston considered the motion of smooth $2$-dimensional manifold $M$ in  isotropic measure preserving Brownian flow $\Phi_t,\ t>0$ in $\mbR^3.$ Authors take point $x\in M$ and consider the mean and Gaussian curvature of
$$
M_t=\Phi_t(M)
$$
at the point
$$
x_t=\Phi_t(x).
$$
The stochastic differential equation for this vector is derived. Using Lyapunov function method the recurrence of obtained diffusion is proved.

C.L. Zirbel and E. Cinlar in \cite{14} studied asymptotic behavior of translation
 $$
 C_t=\frac{1}{M_t(\mbR^d)}\int_{\mbR^d}x M_t(dx)
 $$
and dispersion
$$
D_t=\frac{1}{M_t(\mbR^d)}\int_{\mbR^d}(x-C_t)(x-C_t)^T M_t(dx)
$$
for a measure valued process $M_t,\ t\geq 0$ obtained as an image of probability measure on $\mbR^d$  under isotropic Brownian flow. In the paper authors obtained the representation of the mean of $D_t$ and the covariance of $C_t$ in terms of one dimensional diffusion on $\mbR_+$ and described its asymptotic behavior for different dimensions of spatial variable and sign of the largest Lyapunov exponent of the flow.

In present paper we study evolution of Gaussian field in some random media. We suggest completely new approach based on equation with interaction represented by A.A. Dorogovtsev in \cite{16}, \cite{9}. In comparison with the mentioned works our model describe the evolution of Gaussian field taking into account the motion of all points of our Gaussian field. Our aim is to define  self-intersection local times for evolving Gaussian field and describe its asymptotics. According to main aim of the paper it is organized as follows. In Section 2 we introduce a class of planar Gaussian fields. Conditions on covariance function allowed us to prove that self-intersection local for this class of Gaussian fields exists. Self-intersection local time can be considered as the integral from a weight-function with respect to some random measure. The construction of such measure and its properties are also discussed in Section 2.

In Section 3 we construct a self-intersection local time for an image of planar Gaussian field under some deterministic diffeomorphism. Representing a new delta-family generated by given diffeomorphism  we conclude that self-intersection local time exists.

In Section 4 we prove that obtained results in Section 3 can be extended on random diffeomorphisms. This randomness was assigned by considering the diffeomorphism depending on the fixed number of values of  Gaussian field for which a self-intersection local time exists. It allows us to consider a deterministic equation with interaction with initial measure corresponding to the occupation measure of given Gaussian field. Approximating the occupation measure by the sequence of discrete measures we turn up to the previous case and prove the existence of self-intersection local time for the image of Gaussian field under the solution to deterministic equation with interaction.

In Section 5 we consider the Gaussian field presented in Section 2 evolving in a stochastic flow. Due to Section 2 a self-intersection local time for this Gaussian field exists. Mathematically the evolution is described by an image of  Gaussian field under a stochastic flow obtained as a solution to some SDE. Using arguments of Section 4 one can conclude that a self-intersection local time exists for an image of Gaussian field under a solution to SDE. The main aim of Section 5 is to describe a behavior  of self-intersection local time for an image of Gaussian field under a stochastic flow. This question is solved in Section 5 in the following directions. Firstly we consider the evolution of Gaussian field in isotropic Brownian flow which satisfies the following SDE
$$
\begin{cases}
dx(u,t)=\int_{\mbR^2}\varphi(x(u,t)-p)W(dp,dt)\\	
x(u,0)=u,\ u\in\mbR^2.
\end{cases}
$$
Here a function $\varphi\in C^{\infty}_0(\mbR^2)$ describes the interaction of particles $x(u,t)$ with a random media $W.$ The size of support of function $\varphi$ plays the role of radius of interaction with a media. The first question considered in Section 5 is how a self-intersection local time  changes when the radius of interaction decreases to zero. Intuitively one can expect that with the decreasing of radius of interaction our Gaussian field takes a form of a coil, i.e. a self-intersection local time is increasing. In Section 5 we show that the expectation of self-intersection local time for the image of Gaussian field in $e^{\frac{ck(k-1)t}{2\ve^2}}$ times greater than the expectation of self-intersection local time for the initial Gaussian field. Here $\ve$ is the radius of interaction with a media. For precise definition of parameters see Section 5. Let us fix the radius of interaction with a media and ask how a self-intersection local time changes for large time intervals. Here we establish that the expectation of self-intersection local time for the image of Gaussian field in $e^{\frac{ck(k-1)t}{2}}$ times greater when the expectation of self-intersection local time for the initial Gaussian field.

In the second case stochastic flow obtained as solution to the following equation with interaction
$$
\begin{cases}
dx(u,t)=a(x(u,t)-\int_{\mbR^2}v\mu_t(dv))dt+\int_{\mbR^2}\varphi(x(u,t)-p)W(dp,dt)\\	 x(u,0)=u,\ u\in\mbR^2\\
\mu_t=\mu_0\circ x(\cdot,t)^{-1}.
\end{cases}
$$
The first term of equation describes the motion of particles with respect to joint center of mass $\int_{\mbR^2}v\mu_t(dv).$ Our particles run away from a center of mass. Without stochastic term such behavior of particles implies that the expectation of self-intersection local time for the image of Gaussian field is decreasing. Including a stochastic term in the equation we obtain that the expectation of self-intersection local time for the image of Gaussian field in $e^{(kc-2a)(k-1)t}$ times greater than the expectation of self-intersection local time for the initial Gaussian field, i.e. it exponentially grows to infinity. The apogee of Section 5 is an introduction of right renormalization which for every trajectory of our Gaussian field allows us to conclude the existence of finite limit of self-intersection local time for the image of this Gaussian field under a solution of mentioned equation with interaction when t grows to infinity.

\section{Self-intersection local time for one class of Gaussian fields}

In present section we prove the existence of self-intersection local time for a certain class of planar Gaussian fields. Let us start from a general definition of self-intersection local time. Let $\varsigma(u),\ u\in D\subset\mbR^2$ be a planar random field and $\rho:\mbR^2\to\mbR$ be some measurable function. In the paper we call $\rho$ weight-function or weight. The function $\rho$ can describe the properties of media in which random field moves. Consider a  family of functions which approximates delta-function
\begin{equation}
\label{eq2.100}
f_{\ve}(z)=\frac{1}{2\pi\ve}e^{-\frac{\|z\|^2}{2\ve}},\ z\in\mbR^2.
\end{equation}
Further in the paper we will use the following notation. For any $z_1,\ldots,z_n$
$$
d\vec{z}=dz_1\ldots dz_n.
$$
\begin{defn}
\label{defn2.1} A random variable
$$
\cT^\varsigma_k(\rho)=\int_D\ldots\int_D\prod^{k-1}_{i-1}\delta_0(\varsigma(u_{i+1})-\varsigma(u_{i}))d\vec{u}=
$$
$$
=L_2-\lim_{\ve\to0}\int_{D}\ldots\int_{D}\prod^{k-1}_{i-1}f_{\ve}(\varsigma(u_{i+1})-\varsigma(u_{i}))d\vec{u}
$$
is said to be self-intersection local time for a field $\varsigma$ whenever the limit exists.
\end{defn}
Note that for planar random processes self-intersection local time can not be defined as the limit of approximations. To obtain the existence of finite limit one must construct renormalizations. For planar Wiener process such renormalizations were constructed in works \cite{1}--\cite{3}. In \cite{3} E.B. Dynkin constructed renormalization for bounded measurable $\mbR$- valued weights. In \cite{27} we expanded E.B. Dynlin result for Hilbert-valued weights.  For planar Gaussian integrators we constructed renormalized self-intersection local time in \cite{25}, \cite{26}.

Let
$$
\eta(u)=(\eta_1(u),\eta_2(u)),\ u\in D
$$
be a Gaussian random field, where $D=[0;1]\times[1;2].$ Coordinates $\eta_i:D\to\mbR,\ i=\overline{1,2}$ are
independent equidistributed  centered Gaussian random fields
with the covariance function
$$
E\eta_i(y_1, t_1)\eta_i(y_2, t_2)=e^{-|y_2-y_1|^{\alpha}}t_1\wedge t_2,\ i=\overline{1,2},\ \alpha\in (0;2].
$$
Note that for $t_1<t_2$
$$
E(\eta_1(y_2,t_2)-\eta_1(y_1,t_1))^2\leq
$$
$$
\leq 8(|y_2-y_1|^{\alpha}+(t_2-t_1)).
$$
Since $\eta_1$ is Gaussian, then due to Kolmogorov continuity theorem one can conclude that the field $\eta(y,t),\ y\in[0;1],\ t\in[1;2]$ has a continuous modification with respect to both variables $y$ and $t$ . Moreover
for every fixed $y\in[0;1]$ the processes $\eta_i(y, t),\  t\in[1;2],\ i=1,2$ are independent Brownian motions.
Therefore our random field represents the set of planar Wiener processes indexed by $y\in[0;1].$ A correlation between these planar Wiener processes implies the existence of self-intersection local time for the field $\eta$ despite of the fact that for one planar Wiener process a self-intersection local time does not exist.

\begin{thm}
\label{thm2.1} For a continuous bounded  weight-function $\rho:\mbR^2\to\mbR$
there exists
$$
\cT^{\eta}_k(\rho)=L_2-\lim_{\ve\to0}\int_D\ldots\int_D\rho (\eta(u_1))\prod^{k-1}_{i=1}f_{\ve}(\eta(u_{i+1})-\eta(u_i))d\vec{u}.
$$
\end{thm}
\begin{proof} Denote by
$$
I_{\ve}=\int_D\ldots\int_D\rho (\eta(u_1))\prod^{k-1}_{i=1}f_{\ve}(\eta(u_{i+1})-\eta(u_i))d\vec{u}.
$$
To prove the statement let us check that there exists finite limit
$$
\lim_{\ve_1,\ve_2\to0}EI_{\ve_1}I_{\ve_2}.
$$
For Gaussian random variables $\xi_1,\ldots,\xi_n$
denote by $G(\xi_1,\ldots,\xi_n)$ Gram determinant constructed by random variables $\xi_1,\ldots,\xi_n$ considered as elements of Hilbert space of square integrable random variables.
One can see that
$$
EI_{\ve_1}I_{\ve_2}=E\int_D\ldots\int_D \rho (\eta(u_1))\rho (\eta(v_1))\prod^{k-1}_{i=1}f_{\ve_1}(\eta(u_{i+1})-\eta(u_i))\cdot
$$
$$
\cdot\prod^{k-1}_{i=1}f_{\ve_2}(\eta(v_{i+1})-\eta(v_i))d\vec{u}d\vec{v}=
$$
$$
=\int_D\ldots\int_D\int_{\mbR^{2k}}\rho(z_1)\rho(z_{k+1})\prod^{k}_{i=2}f_{\ve_1}(z_i)\prod^{2k}_{j=k+2}f_{\ve_2}(z_j)\cdot
$$
$$
\cdot \frac{1}{(2\pi)^k G(\eta_1(u_1),\Delta\eta_1(u_1),\ldots,\eta_1(v_1),\ldots,\Delta\eta_1(v_{k-1}))}e^{-\frac{(A^{-1}\vec{z},\vec{z})}{2}}d\vec{z}d\vec{u}d\vec{v},
$$
where $A$ is Gramian matrix constructed by random variables
$$
\eta_1(u_1),\ldots,\Delta\eta_1(u_{k-1}),\ldots,\Delta\eta_1(v_{k-1}),\eta_2(u_1),\ldots,\eta_2(v_{k-1}).
$$
and
$$
\Delta\eta_1(u_i)=\eta_1(u_{i+1})-\eta_1(u_i).
$$
Denote by
$$
p(\vec{z})=\frac{1}{(2\pi)^k G(\eta_1(u_1),\Delta\eta_1(u_1),\ldots,\eta_1(v_1),\ldots,\Delta\eta_1(v_{k-1}))}e^{-\frac{(A^{-1}\vec{z},\vec{z})}{2}}.
$$
Note that
$$
\int_{\mbR^{2k}}\rho(z_1)\rho(z_{k+1})\prod^{k}_{i=2}f_{\ve_1}(z_i)\prod^{2k}_{j=k+2}f_{\ve_2}(z_j)p(\vec{z})d\vec{z}\to
$$
$$
\int_{\mbR^4}\rho(z_1)\rho(z_{k+1})p(z_1,0,\ldots,0,z_{k+1},0,\ldots,0)dz_1dz_{k+1}
$$
as $\ve_1,\ve_2\to0.$ Moreover using boundedness of weight $\rho$ one can conclude that
$$
\Big|\int_{\mbR^{2k}}\rho(z_1)\rho(z_{k+1})\prod^{k}_{i=2}f_{\ve_1}(z_i)\prod^{2k}_{j=k+2}f_{\ve_2}(z_j)p(\vec{z})d\vec{z}\Big|\leq
$$
$$
\leq c_1\Big|\int_{\mbR^{2k}}\prod^{k}_{i=2}f_{\ve_1}(z_i)\prod^{2k}_{j=k+2}f_{\ve_2}(z_j)p(\vec{z})d\vec{z}\Big|,\ c_1>0.
$$
Integrating the function $p$ with respect to variables $z_1$ and $z_{k+1}$ one can obtain the density of Gaussian vector
$$
(\Delta\eta_1(u_1),\ldots,\Delta\eta_1(u_{k-1}),\Delta\eta_1(v_1),\ldots,\Delta\eta_1(v_{k-1}),\Delta\eta_2(u_1),\ldots,\Delta\eta_2(v_{k-1}))
$$
which has the representation
$$
\tilde{p}(\vec{z})=\frac{1}{(2\pi)^{\frac{2k-2}{2}}G(\Delta\eta_1(u_1),\ldots,\Delta\eta_1(v_{k-1}))}e^{-\frac{(\tilde{A}^{-1}\vec{z},\vec{z})}{2}},
$$
where $\tilde{A}$ is Gramian matrix constructed by random variables
$$
\Delta\eta_1(u_1),\ldots,\Delta\eta_1(v_{k-1}),\Delta\eta_2(u_1),\ldots,\Delta\eta_2(v_{k-1})).
$$
One can check that the following relation holds
$$
\Big|\int_{\mbR^{2k-2}}\prod^{k}_{i=2}f_{\ve_1}(z_i)\prod^{2k}_{j=k+2}f_{\ve_2}(z_j)\tilde{p}(\vec{z})d\vec{z}\Big|\leq
$$
$$
\leq \frac{c_2}{G(\Delta\eta_1(u_1),\ldots,\Delta\eta_1(v_{k-1}))},\ c_2>0.
$$
To apply Lebesgue's dominated convergence theorem let us prove that
$$
\int_{D}\ldots\int_D\frac{1}{G(\Delta\eta_1(u_1),\ldots,\Delta\eta_1(v_{k-1}))}d\vec{u}d\vec{v}<+\infty.
$$
Suppose that
$$
u_i=(y_i,t_i),\ v_i=(z_i,s_i),\ i=\overline{1,k}.
$$
Then
$$
G(\Delta\eta_1(u_1),\ldots,\Delta\eta_1(u_{k-1}),\Delta\eta_1(v_{1}),\ldots,\Delta\eta_1(v_{k-1}))=
$$
\begin{equation}
\label{eq2.1}
=G(\Delta \eta_1(y_{1},t_1),\ldots,\Delta \eta_1(y_{k-1},t_{k-1}), \Delta\eta_1(z_1,s_1),\ldots,\Delta\eta_1(z_{k-1},s_{k-1})).
\end{equation}
To obtain lower estimate for \eqref{eq2.1} we need the following lemma.
\begin{lem}
\label{lem2.2}
Let $e_1,\ldots,e_k$ be Gaussian random variables. Then
$$
G(e_2-e_1,e_3-e_2,\ldots,e_k-e_{k-1})\geq
$$
$$
\geq Var(e_2|e_1) Var(e_3|e_1,e_2)\ldots Var(e_k|e_1,e_2,\ldots,e_{k-1}).
$$
\end{lem}
\begin{proof}
Note that
$$
G(e_2-e_1,e_3-e_2,\ldots,e_k-e_{k-1})=G(e_2-e_1,e_3-e_1,\ldots,e_k-e_{1})=
$$
$$
=E(e_2-e_1)^2 Var(e_3-e_1|e_2-e_1)\ldots Var(e_k-e_1|e_2-e_1,\ldots,e_{k-1}-e_{1}).
$$
Since for $K_1\subseteq K_2\subset L_2(\Omega,\cF,P)$ and Gaussian random variable $g$
$$
\inf_{e\in K_1}E(g-e)^2\geq \inf_{e\in K_2}E(g-e)^2,
$$
then
$$
E(e_2-e_1)^2 Var(e_3-e_1|e_2-e_1)\ldots Var(e_k-e_1|e_2-e_1,\ldots,e_{k-1}-e_{1})\geq
$$
$$
\geq Var(e_2|e_1) Var(e_3|e_1,e_2)\ldots Var(e_k|e_1,e_2,\ldots,e_{k-1})
$$
which completes the proof of lemma.
\end{proof}
It follows from Lemma \ref{lem2.2} that
$$
G(\Delta \eta_1(y_{1},t_1),\ldots,\Delta \eta_1(y_{k-1},t_{k-1}), \Delta\eta_1(z_1,s_1),\ldots,\Delta\eta_1(z_{k-1},s_{k-1}))\geq
$$
$$
\geq\sigma^2(y_1,t_1,u_2,t_2)\ldots\sigma^2(y_1,t_1,\ldots,z_k,s_k),
$$
where
$$
\sigma^2(y_1,t_1,\ldots,z_i,s_i)=Var(\eta_1(v_i,s_i)|\eta_1(u_1,t_1),\eta_1(u_2,t_2),\ldots,\eta_1(v_{i-1},t_{i-1})).
$$
In the beginning let us estimate from below
$$
Var\Big(\eta_1(r, t)|\eta_1(p, s),  |p-r|\geq\delta_1, |t-s|\geq\delta_2\Big)
$$
for $0<\delta_1, \delta_2<\frac{1}{2}.$
Since $\eta_1$ is Gaussian random field, then
$$
Var\Big(\eta_1(r, t)|\eta_1(p, s),  |r-p|\geq\delta_1, |t-s|\geq\delta_2\Big)=
$$
$$
=E(\eta_1(r, t)-E(\eta_1(r, t)|\eta_1(p, s), |r-p|\geq\delta_1, |t-s|\geq\delta_2))^2.
$$
Consider non Gaussian random field $\xi(u, s)=\zeta(u)w(s),\ u\in\mbR,\  s\in[0; 1],$ where $\zeta$ is a centered stationary process with the covariance function $e^{-|u|^{\alpha}},\ \alpha\in(0;2]$ and $w$ is standard Wiener process. Suppose that $\zeta$ and $w$ are independent. Note, that $\xi$ and $\eta_1$ have the same covariance function. Really,
$$
E\xi(r_1, t_1)\xi(r_2, t_2)=e^{-|r_1-r_2|^{\alpha}}t_1\wedge t_2.
$$
For Gaussian random field $\eta_1$ a subspace of random variables that is measurable with respect to  $\sigma$-field generated by $\{\eta_1(p, s), |r-p|\geq\delta_1, |t-s|\geq\delta_2\}$ coincides with a closure in $L_2(\Omega,\cF,P)$ of linear span generated by $\{\eta_1(p, s), |r-p|\geq\delta_1, |t-s|\geq\delta_2\}.$ For non Gaussian field $\xi$ a subspace of random variables that is measurable with respect to $\sigma$-field generated by $\{\xi(p, s), |r-p|\geq\delta_1, |t-s|\geq\delta_2\}$ contains a closure in $L_2(\Omega,\cF,P)$ of a linear span generated by $\{\xi(p, s), |r-p|\geq\delta_1, |t-s|\geq\delta_2\}.$ Consequently
$$
E(\eta_1(r, t)-E(\eta_1(r, t)|\eta_1(p, s), |r-p|\geq\delta_1, |t-s|\geq\delta_2))^2\geq
$$
$$
\geq
E(\xi(r, t)-E(\xi(r, t)|\xi(p, s), |r-p|\geq\delta_1, |t-s|\geq\delta_2))^2\geq
$$
$$
\geq
E(\xi(r, t)-E(\xi(r, t)|\zeta(p), |r-p|\geq\delta_1, w(s), |t-s|\geq\delta_2))^2=
$$
$$
=E(\zeta(r)w(t)-\hat{\zeta}(r)\hat{w}(t))^2.
$$
Here
$$
\hat{\zeta}(r)=E(\zeta(r)|\zeta(p), |r-p|\geq\delta_1),
$$
$$
\hat{w}(t)=E(w(t)|w(s), |t-s|\geq\delta_2).
$$
Note that
$$
E(\zeta(r)w(t)-\hat{\zeta}(r)\hat{w}(t))^2=
$$
$$
=E((\zeta(r)-\hat{\zeta}(r)+\hat{\zeta}(r))w(t)-\hat{\zeta}(r)(\hat{w}(t)-w(t)+w(t)))^2=
$$
$$
=E((\zeta(r)-\hat{\zeta}(r))w(t)+(w(t)-\hat{w}(t))\hat{\zeta}(r))^2=
$$
$$
=tE(\zeta(r)-\hat{\zeta}(r))^2+2E(\zeta(r)-\hat{\zeta}(r))w(t)(w(t)-\hat{w}(t))\hat{\zeta}(r)+
$$
\begin{equation}
\label{eq2.9}
+E\hat{\zeta}(r)^2E(w(t)-\hat{w}(t))^2.
\end{equation}
Since
$$
E(\zeta(r)-\hat{\zeta}(r))\hat{\zeta}(r)=0,
$$
then \eqref{eq2.9} equals
$$
tE(\zeta(r)-\hat{\zeta}(r))^2+E\hat{\zeta}(r)^2E(w(t)-\hat{w}(t))^2.
$$
Let us check that there exists a constant $K>0$ such that
\begin{equation}
\label{eq2.12}
Var(\zeta(r)|\zeta(p),\ |r-p|\geq \delta_1)=E(\zeta(r)-\hat{\zeta}(r))^2\geq K\delta^{\alpha+1}_1.
\end{equation}
To prove \eqref{eq2.12} we need the following Lemma.
\begin{lem}
\label{lem2.1}\cite{7} Let $X(t)$ be a stationary process with spectral distribution function $F(\lambda).$ Assume that the absolutely continuous part $F(\lambda)$ is such that there exists a function $\phi$ for which
$$
\frac{dF\Big(\frac{\lambda}{t}\Big)}{\phi(t)}\geq h(\lambda)d\lambda
$$
for all $t<t_0,$ where $h$ is deacreasing on $[0;\infty)$ and
$$
\int^{\infty}_{0}\frac{\ln h(\lambda)}{1+\lambda^2}d\lambda>-\infty.
$$
Then there exists a positive constant $K$ such that
$$
Var(X(0)|X(s),\ |s|\geq t)\geq K\phi(t).
$$
\end{lem}
Since $e^{-|u|^{\alpha}},\ \alpha\in(0;2]$ is characteristic function of symmetric stable distribution with parameter $\alpha,$ then stationary process
$\zeta$ has spectral density $p$ such that
\begin{equation}
\label{eq2.13}
p(y)=p(-y)\sim \frac{c_\alpha}{y^{\alpha+1}},\ y\to+\infty.
\end{equation}
Positivity and continuity of function $p$ on $\mbR$ and tails asymptotics \eqref{eq2.13} imply that there exists a constants $c>0$ such that for any $y\in\mbR$
$$
p(y)>\frac{c}{1+y^{\alpha+1}}.
$$
Then
\begin{equation}
\label{eq2.14}
p\Big(\frac{\lambda}{t}\Big)>c\ \frac{t^{\alpha+1}}{t^{\alpha+1}+\lambda^{\alpha+1}}.
\end{equation}
It follows from \eqref{eq2.14} that for all $t<1$
$$
p\Big(\frac{\lambda}{t}\Big)>c\ \frac{t^{\alpha+1}}{1+\lambda^{\alpha+1}}.
$$
Consequently one can apply Lemma \ref{lem2.1} for the stationary process $\zeta$ with
$$
\phi(t)=t^{\alpha+1},
$$
$$
h(\lambda)=\frac{1}{1+\lambda^{\alpha+1}}.
$$
Note that $h$ is decreasing function on $[0;\infty)$. Moreover one can check that
$$
\int^{\infty}_0\frac{\ln\frac{1}{1+\lambda^{\alpha+1}}}{1+\lambda^2}d\lambda>-\infty.
$$
It follows from Lemma \ref{lem2.1} that there exists $K>0$ such that
$$
Var (\zeta(r)|\zeta(p),\ |r-p|\geq \delta_1)\geq
$$
$$
\geq K \phi(r-p)\geq K\delta^{\alpha+1}_1.
$$
Also one can see that
$$
E\hat{\zeta}(r)^2\leq E\zeta(r)^2=1,
$$
$$
E(w(t)/w(s), |t-s|\geq\delta_2)=\frac{1}{2}(w(t-\delta_2)+w(t+\delta_2))
$$
and
$$
E(w(t)-\hat{w}(t))=\frac{1}{2}\delta_2.
$$
Consequently the conditional variance can be estimated as follows
\begin{equation}
\label{eq2.10}
V(\eta_1(r, t)|\eta_1(p, s), |p-r|\geq\delta_1, |t-s|\geq\delta_2)\geq K\ t\ \delta^{\alpha+1}_1+\frac{1}{2}\delta_2,\ K>0.
\end{equation}
It follows from \eqref{eq2.10}
that
$$
Var\Big(\eta_1(z_i,s_i)|\eta_1(y_1,t_1),\eta_1(y_2,t_2),\ldots,\eta_1(z_{i-1},s_{i-1})\Big)\geq
$$
$$
\geq Var\Big(\eta_1(z_i,s_i)|\eta_1(p,\tau),
$$
$$
|z_i-p|\geq\min_{q\in\{y_1,\ldots,z_{i-1}\}}|z_i-q|,\ |s_i-\tau|\geq\min_{r\in\{t_1,\ldots,s_{i-1}\}}|s_i-r|\Big)\geq
$$
\begin{equation}
\label{eq2.1000}
\geq K s_i(\min_{q\in\{y_1,\ldots,z_{i-1}\}}|z_i-q|)^{\alpha+1}+\frac{1}{2}\min_{r\in\{t_1,\ldots,s_{i-1}\}}|s_i-r|.
\end{equation}
Note that
$$
\frac{1}{K s_i(\min_{q\in\{y_1,\ldots,z_{i-1}\}}|z_i-q|)^{\alpha+1}+\frac{1}{2}\min_{r\in\{t_1,\ldots,s_{i-1}\}}|s_i-r|}\leq
$$
$$
\leq \sum_{q\in\{y_1,\ldots,z_{i-1}\},\ r\in\{t_1,\ldots,s_{i-1}\}}\frac{1}{K s_i|z_i-q|^{\alpha+1}+\frac{1}{2}|s_i-r|}.
$$
Therefore
$$
\int_D\ldots\int_D
\frac{1}{G(\Delta \eta_1(y_{1},t_1),\ldots,\Delta \eta_1(z_{k-1},s_{k-1}) )}d\vec{y}d\vec{t}d\vec{z}d\vec{s}\leq
$$
$$
\leq\sum_{q_1\in\{y_1,y_2\},\ r_1\in\{t_1,t_2\}}\ldots\sum_{q_{2k-1}\in\{y_1,\ldots,z_{k-1}\},\ r_{2k-1}\in\{t_1,\ldots,s_{k-1}\}}
$$
$$
\int_D\ldots\int_D\frac{1}{K\ t_2|y_2-y_1|^{\alpha+1}+\frac{1}{2}|t_2-t_1|}\cdot\frac{1}{K\ t_3|y_3-q_1|^{\alpha+1}+\frac{1}{2}|t_3-r_1|}\ldots
$$
$$
\ldots\frac{1}{K\ s_k|z_k-q_{2k-1}|^{\alpha+1}+\frac{1}{2}|s_k-r_{2k-1}|}d\vec{y}d\vec{t}d\vec{z}d\vec{s}.
$$
To finish the proof it suffices to check that there exists a constant $c_3>0$ such that
$$
\sup_{q_{2k-1}\in[0;1]}\sup_{r_{k-1}\in[1;2]}\Big|\int^1_0\int^{2}_{r_{2k-1}}\frac{1}{2K|z_{k}-q_{2k-1}|^{\alpha+1}+s_{k}-r_{2k-1}}ds_kdz_{k}\Big|<c_3.
$$
Really,
$$
\Big|\int^1_0\int^{2}_{r_{2k-1}}\frac{1}{2K|z_{k}-q_{2k-1}|^{\alpha+1}+s_{k}-r_{2k-1}}ds_kdz_{k}\Big|=
$$
$$
=\Big|\int^1_0\ln 2K(|z_{k}-q_{2k-1}|^{\alpha+1}+2-r_{2k-1})dz_k\Big|+
$$
$$
+\Big|\int^1_{0}\ln 2K |z_{k}-q_{2k-1}|^{\alpha+1} dz_k\Big|< c_3,\ c_3>0.
$$
\end{proof}
\begin{remk}
\label{remk2.1}By truncation arguments $\cT^{\eta}_k(\rho)$ can be defined for continuous unbounded weights on $\mbR^2.$
\end{remk}
Obtained value $\cT^{\eta}_k(\rho)$ can be considered as an integral from the function $\rho$ with respect to random measure $\nu_k$
which can called by self-intersection measure of order $k.$ The existence and construction of such measure is described in the following Lemma.
\begin{lem}
\label{lem2.6} For an arbitrary $k\geq2$ there exists a finite measure $\nu_k$ on $\mbR^2$ such that for any $A\in\cB(\mbR^2)$
$$
\nu_k(A)=\int_D\ldots\int_D1_{A}(\eta(u_1))\prod^{k-1}_{j=1}\delta_0(\eta(u_{i+1})-\eta(u_{i}))d\vec{u}.
$$
\end{lem}
\begin{proof}
Let $\cP$ be a set of polynomials $\cP$ with rational coefficients. Consider a function
 $$
 \psi(y)=
 \begin{cases}
 Ce^{-\frac{1}{1-\|y\|^2}},\ \|y\|<1,\\
 0,\ \ \ \ \ \ \ \ \ \ \ \|y\|\geq 1,
 \end{cases}
 $$
 where $C$ is positive constant such that
$$
\int_{\mbR^2}\psi(y)dy=1.
$$
For $R>3$ put
$$
h^R(y)=\int_{\mbR^2}1_{[-R+2;R-2]^2}(z)\psi(y-z)dz.
$$
Note that $h^R\in C^{\infty}(\mbR^2).$
It follows from Theorem \ref{thm2.1} that for any $p\in\cP$ there exists
$$
\int_D\ldots\int_D p(\eta(u_1))h^R(\eta(u_1))\prod^{k-1}_{j=1}\delta_0(\eta(u_{i+1})-\eta(u_{i}))d\vec{u}=
$$
$$
=L_{2}-\lim_{\ve\to0}\int_D\ldots\int_D p(\eta(u_1))h^R(\eta(u_1))\prod^{k-1}_{j=1}f_{\ve}(\eta(u_{i+1})-\eta(u_{i}))d\vec{u}.
$$
There exist $\tilde{\Omega},\ P(\tilde{\Omega})=1$ and
subsequence $\{\ve_n,\ n\geq1\},\ \ve_n\to0,\ n\to+\infty$  such that for any $\omega\in\tilde{\Omega}$ and for any $p\in\cP$
$$
\Phi(p)=\int_D\ldots\int_D p(\eta(u_1))h^R(\eta(u_1))\prod^{k-1}_{j=1}\delta_0(\eta(u_{i+1})-\eta(u_{i}))d\vec{u}=
$$
$$
=\lim_{n\to+\infty}\int_D\ldots\int_D p(\eta(u_1))h^R(\eta(u_1))\prod^{k-1}_{j=1}f_{\ve_n}(\eta(u_{i+1})-\eta(u_{i}))d\vec{u}.
$$
Note that for any $\omega\in\tilde{\Omega},\ p_1,p_2\in\cP,\ r_1,r_2\in\mbQ$
$$
\Phi(r_1p_1+r_2p_2)=r_1\Phi(p_1)+r_2\Phi(p_2).
$$
Let us check that nonnegative linear functional $\Phi$ defined on $\cP$  is bounded, i.e. for any $\omega\in\tilde{\Omega}$
\begin{equation}
\label{eq2.11}
\|\Phi\|<+\infty,
\end{equation}
where
$$
\|\Phi\|=\sup_{p\in B(0,1)}|\Phi(p)|.
$$
Here
$$
B(0;1)=\Big\{\rho\in \cP:\sup_{[-R;R]^2}|p(z)|\leq 1\Big\}.
$$
Notice that for any $\omega\in\Omega$
$$
\int_D\ldots\int_D p(\eta(u_1))h^R(\eta(u_1))\prod^{k-1}_{j=1}\delta_0(\eta(u_{i+1})-\eta(u_{i}))d\vec{u}\leq
$$
$$
\leq \sup_{v\in[-R;R]^2}|p(v)|\int_D\ldots\int_D h^R(\eta(u_1))\prod^{k-1}_{j=1}\delta_0(\eta(u_{i+1})-\eta(u_{i}))d\vec{u}
$$
which proves the boundness of $\Phi.$ Therefore $\Phi$ nonnegative linear bounded functional on $\cP$ which can be extended to continuous linear functional on $C([-R;R]^2,\mbR).$ Consequently there exists a measure $\nu^{R}_k$ on $\cB([-R;R]^2)$ such that for any $f\in C([-R;R]^2,\mbR)$
$$
\Phi(f)=\int_{\mbR^2}f(u)\nu^{R}_k(du).
$$
Note that
$$
\nu^{R}_k([-R;R]^2)\leq
$$
$$
\leq\int_D\ldots\int_D\prod^{k-1}_{i=1}\delta_0(\eta(u_{i+1})-\eta(u_{i}))d\vec{u}.
$$
Consider
$$
f\in C_0(\mbR^2,\mbR),
$$
where $C_0(\mbR^2,\mbR)$ is the space of continuous functions $f:\mbR^2\to\mbR$ with a compact support. There exists $R>0$ such that $supp f\subset[-R+3;R-3]^2,$ then
$$
\int_{\mbR^2}f(u)\nu^{R}_k(du)=
$$
$$
\int_{D}\ldots\int_{D}f(\eta(u_1))\prod^{k-1}_{j=1}\delta_0(\eta(u_{i+1})-\eta(u_{i}))d\vec{u}.
$$
Moreover for $R^{\prime}>R$
$$
\int_{\mbR^2}f(u)\nu^{R^{\prime}}_k(du)=\int_{\mbR^2}f(u)\nu^{R}_k(du).
$$
Consequently
$$
\Phi(f)=\int_{\mbR^2}f(u)\nu^R_k(du)\to\int_{\mbR^2}f(u)\nu_k(du),\ R\to+\infty.
$$
Therefore there exists a measure $\nu_k$ on $\cB(\mbR^2)$ defined by the values of integrals from continuous functions on $\mbR^2$ with a compact support. Moreover
$$
\int_{\mbR^2}1\ \nu_k(du)=\lim_{R\to+\infty}\int_{\mbR^2}1\ \nu^R_k(du)=
$$
$$
=\int_{D}\ldots\int_{D}\prod^{k-1}_{j=1}\delta_0(\eta(u_{i+1})-\eta(u_{i}))d\vec{u}.
$$
\end{proof}
The following statements describe properties of measure $\nu_k$ which are important for our future considerations.
Let us check that the random variable $\nu_k(\mbR^2)$ has all moments.
\begin{lem}
\label{lem2.7}For any $m\geq 1$
$$
E(\nu_k(\mbR^2))^m<+\infty.
$$
\end{lem}
\begin{proof}
Using the definition of measure $\nu_k$ one can write
$$
E(\nu_k(\mbR^2))^m=
$$
$$
E\Big(\int_{D}\ldots\int_{D}\prod^{k-1}_{i=1}\delta_0(\eta(u^1_{i+1})-\eta(u^1_{i}))\prod^{k-1}_{i=1}\delta_0(\eta(u^2_{i+1})-\eta(u^2_{i}))\ldots
$$
\begin{equation}
\label{eq2.17}
\ldots \delta_0(\eta(u^m_{i+1})-\eta(u^m_{i}))\Big)d\vec{u}.
\end{equation}
It follows from \cite{5} that \eqref{eq2.17} is finite if the following integral $$
\int_{D}\ldots\int_{D}\frac{1}{G(\Delta \eta_1(u^1_1),\ldots,\Delta\eta_1(u^1_{k-1}),\ldots,\Delta \eta_1(u^m_1),\ldots,\Delta\eta_1(u^m_{k-1}))}d\vec{u}
$$
converges.
It follows from Lemma \ref{lem2.1} that
$$
G(\Delta \eta_1(u^1_1),\ldots,\Delta\eta_1(u^1_{k-1}),\ldots,\Delta \eta_1(u^m_1),\ldots,\Delta\eta_1(u^m_{k-1}))\geq
$$
$$
\geq Var(\eta_1(u^1_2)|\eta_1(u^1_1))\ldots Var(\eta_1(u^1_k)|\eta_1(u^1_1,\ldots,\eta_1(u^1_{k-1}))\cdot
$$
$$
\cdot Var(\eta_1(u^2_1)|\eta_1(u^1_1),\ldots,\eta_1(u^1_{k-1}),\eta_1(u^2_{1}))\ldots
$$
$$
\ldots Var(\eta_1(u^m_k)|\eta_1(u^1_1),\ldots,\eta_1(u^1_{k-1}),\eta_1(u^2_{1}),\ldots,\eta_1(u^m_{k-1})) .
$$
Consequently
$$
\int_{D}\ldots\int_{D}\frac{1}{G(\Delta \eta_1(u^1_1),\ldots,\Delta\eta_1(u^1_{k-1}),\ldots,\Delta \eta_1(u^m_1),\ldots,\Delta\eta_1(u^m_{k-1}))}d\vec{u}\leq
$$
$$
\leq\int_{D}\ldots\int_{D}\frac{1}{Var(\eta_1(u^1_2)|\eta_1(u^1_1))}\ldots\frac{1}{Var(\eta_1(u^1_k)|\eta_1(u^1_1,\ldots,\eta_1(u^1_{k-1}))}\cdot
$$
$$
\cdot \frac{1}{Var(\eta_1(u^2_1)|\eta_1(u^1_1),\ldots,\eta_1(u^1_{k-1}),\eta_1(u^2_{1}))}\ldots
$$
$$
\ldots \frac{1}{Var(\eta_1(u^m_k)|\eta_1(u^1_1),\ldots,\eta_1(u^1_{k-1}),\eta_1(u^2_{1}),\ldots,\eta_1(u^m_{k-1}))}d\vec{u}.
$$
Applying estimate \eqref{eq2.10} and reaping arguments of proof of Theorem \ref{thm2.1} one can easily finish the proof of lemma.
\end{proof}
Next property of measure $\nu_k$ is important for investigation of asymptotics of self-intersection local time for the image of field $\eta$ under a solution to equation with interaction and will be applied in Section 5.
\begin{lem}
\label{lem2.8}
$$
E\int_{\mbR^4}\ln_+\frac{1}{\|u-v\|}\nu_k(du)\nu_k(dv)<+\infty.
$$
\end{lem}
\begin{proof}
Note that due to the definition of measure $\nu_k$
$$
E\int_{\mbR^4}\ln_+\frac{1}{\|u-v\|}\nu_k(du)\nu_k(dv)=
$$
$$
=E\int_D\ldots\int_D\ln_+\frac{1}{\|\eta(u)-\eta(v)\|}\prod^{k-1}_{i=1}\delta_0(\eta(u_{i+1})-\eta(u_{i}))\cdot
$$
\begin{equation}
\label{eq2.18}
\cdot\prod^{k-1}_{i=1}\delta_0(\eta(v_{i+1})-\eta(v_{i}))d\vec{u}d\vec{v}.
\end{equation}
Let $P$ be a projection in $L_2(\Omega,\cF,P)$ on linear subspace generated by $\eta(u_1),\eta(v_1).$ Then
$$
\Delta\eta(u_i)=P\Delta\eta(u_i)+\zeta(u_i),
$$
where $\zeta(u_i)=\Delta\eta(u_i)-P\Delta\eta(u_i).$ Taking conditional expectation with respect to $\eta(u_1),\eta(v_1)$ one can write that \eqref{eq2.18} less or equal to
$$
E\int_D\ldots\int_D\ln_+\frac{1}{\|\eta(u)-\eta(v)\|}\cdot
$$
$$
\cdot\frac{1}{G(\zeta(u_1),\ldots,\zeta(u_{k-1}),\zeta(v_1),\ldots,\zeta(v_{k-1}))}d\vec{u}d\vec{v}=
$$
$$
=E\ E\Big(\int_D\ldots\int_D\ln_+\frac{1}{\|\eta(u_1)-P_{\eta(u_1)}\eta(v_1)-\zeta_{\eta(u_1)}(v_1)\|}\cdot
$$
$$
\cdot\frac{1}{G(\zeta(u_1),\ldots,\zeta(u_{k-1}),\zeta(v_1),\ldots,\zeta(v_{k-1}))}d\vec{u}d\vec{v}|\eta(u_1)\Big)=
$$
$$
=E\int_D\ldots\int_D\int_{\mbR^2}\ln_+\frac{1}{\|y\|}\ \frac{1}{2\pi\|\zeta_{\eta(u_1)}(v_1)\|^2}e^{-\frac{\|y-\eta(u_1)+P_{\eta(u_1)}\eta(v_1)\|^2}{2\|\zeta_{\eta(u_1)}(v_1)\|^2
}}
$$
$$
\cdot\frac{1}{G(\zeta(u_1),\ldots,\zeta(u_{k-1}),\zeta(v_1),\ldots,\zeta(v_{k-1}))}dyd\vec{u}d\vec{v}\leq
$$
$$
\leq E\int_D\ldots\int_D\int^1_0\ln_+\frac{1}{r}dr\ \frac{1}{2\pi\|\zeta_{\eta(u_1)}(v_1)\|^2}\cdot
$$
$$
\cdot\frac{1}{G(\zeta(u_1),\ldots,\zeta(u_{k-1}),\zeta(v_1),\ldots,\zeta(v_{k-1}))}d\vec{u}d\vec{v}.
$$
Here $P_{\eta(u_1)}$ is a projection in $L_2(\Omega,\cF,P)$ on linear subspace generated by $\eta(u_1)$ and
$$
\zeta_{\eta(u_1)}(v_1)=\eta(v_1)-P_{\eta(u_1)}\eta(v_1).
$$
Let us check that
$$
\int_D\ldots\int_D \frac{1}{\|\zeta_{\eta(u_1)}(v_1)\|^2}\cdot
$$
$$
\cdot\frac{1}{G(\zeta(u_1),\ldots,\zeta(u_{k-1}),\zeta(v_1),\ldots,\zeta(v_{k-1}))}d\vec{u}d\vec{v}<+\infty.
$$
Really
$$
\|\zeta_{\eta(u_1)}(v_1)\|^2=E\Big(\eta(v_1)-P_{\eta(u_1)}\eta(v_1)\Big)^2\geq
$$
$$
\geq Var(\eta(v_1)|\eta(u_1),\ldots,\eta(u_{k-1})).
$$
Moreover
$$
Var\Big(\zeta(u_{i})|\zeta(u_1),\ldots,\zeta(u_{i-1})\Big)\geq
$$
$$
\geq Var\Big(\eta(u_{i+1})|\eta(u_1),\ldots,\eta(u_{i-1},\eta(u_i),\eta(v_1))\Big).
$$
Applying estimate \eqref{eq2.1000} and reaping arguments of proof of Theorem \ref{thm2.1} one can complete the proof.
\end{proof}
\section{Self-intersection local time for the image of Gaussian field under deterministic diffeomorphism}
\label{section3}
Consider the equation with interaction defined in Introduction
\begin{equation}
\label{eq3.1}
\begin{cases}
dx(u,t)=a(x(u,t),\mu_t)dt+\int_{\mbR^2}b(x(u,t),\mu_t,q)W(dt,dq)\\	
x(u,0)=u,\ u\in\mbR^2\\
\mu_t=\mu_0\circ x(\cdot,t)^{-1}.
\end{cases}
\end{equation}
Suppose that coefficients are Lipschitz functions with respect to spatial and measure-valued variable. Then there exists unique solution to equation with interaction \eqref{eq3.1} \cite{16} (p. 55). Moreover if coefficients are two times continuously differentiable with respect to spatial variable, then under a fixed $t$ the solution is a diffeomorphism \cite{16} (p. 66). Let $\mu_0$ be an occupation measure of Gaussian field $\eta(u),\ u\in D$ defined in Section 2, i.e.
$$
\mu_0(A)=\int_D1_{A}(\eta(u))du,\ A\in\cB(\mbR^2).
$$
Then $\mu_t$ is the occupation measure of the field $x(\eta(u),t),\ u\in D.$ Our main goal in this article is to prove the existence of self-intersection local time for the field $x(\eta(u),t),\ u\in D$ and describe its asymptotics as $t\to+\infty.$ Let us start from the existence of self-intersection local time. The field $x(\eta(u),t),\ u\in D$ is the image of Gaussian field $\eta(u),\ u\in D$ under the random diffeomorphism $x(\cdot,t):\mbR^2\to\mbR^2$ which depends on $\eta.$ To understand why a self-intersection local time exists consider a case when the diffeomorphism $x(\cdot,t)$ is deterministic.

Consider a field $F(\eta(u)),\ u\in D,$ where $F:\mbR^2\to\mbR^2$ is a deterministic diffeomorphism. To define an approximating family $\cT^{F(\eta)}_{\ve,k}$ we use a delta family of functions introduced in \cite{27}
$$
f^{F}_{\ve}(v_1,\ldots,v_{k})=\frac{1}{|\det F^{\prime}(F^{-1}(v_1))|^{k-1}}\prod^{k-1}_{i=1}f_{\ve}(F^{-1}(v_{i+1})-F^{-1}(v_{i})),
$$
where $f_{\ve}$ from \eqref{eq2.100}.
It can be checked that the family $\{f^{F}_{\ve},\ \ve>0\}$ approximates the delta function, i.e. for any $v_k\in\mbR^2,\ \varphi\in C_{b}(\mbR^{2(k-1)})$
$$
\int_{\mbR^2(k-1)}\varphi(v_1,\ldots,v_{k-1})f^{F}_{\ve}(v_1,\ldots,v_k)d\vec{v}\to \varphi(v_k,\ldots,v_k)
$$
as $\ve\to 0.$ For proof see \cite{27}.
Hence an approximating family for $\cT^{F(\eta)}_{\ve,k}$  can have the following representation
$$
\cT^{F(\eta)}_{\ve,k}=\int_D\ldots\int_D f^{F}_{\ve}(F(\eta(u_{1})),\ldots,F(\eta(u_{k}))d\vec{u}=
$$
$$
=\int_D\ldots\int_D \frac{1}{|\det F^{\prime}(\eta(u_1))|^{k-1}}\prod^{k-1}_{i=1}f_{\ve}(\eta(u_{i+1})-\eta(u_{i}))d\vec{u}=
$$
$$
=\cT^{\eta}_{\ve,k}(\frac{1}{|\det F^{\prime}|^{k-1}}).
$$
Applying Theorem \ref{thm2.1} and Remark \ref{remk2.1} one can conclude the following statement.
\begin{thm}
\label{thm3.1}
There exists
$\cT^{F(\eta)}_{k}=L_2\mbox{-}\lim_{\ve\to0}\cT^{F(\eta)}_{\ve, k}$
and
$$
\cT^{F(\eta)}_{k}=\cT^{\eta}_{k}\Big(\frac{1}{|\det F^{\prime}|^{k-1}|}\Big).
$$
\end{thm}
Let us apply Theorem \ref{thm3.1} to understand asymptotics of $\cT^{x(\eta,t)}_k$ in the following case.
\begin{expl}
\label{expl3.1}
Let $\eta(u),\ u\in D$ be a Gaussian random field defined in Section 2 and $\mu_0$ be an occupation measure of $\eta.$ Consider the following deterministic equation with interaction
$$
	\begin{cases}
		dx(u,t)=(Ax(u,t)+\int_{\mbR^2}v\mu_t(dv))dt\\	
		x(u,0)=u\\
		\mu_t=\mu_0\circ x(\cdot,t)^{-1},
		\end{cases}
$$
where $x(\cdot,t):\mbR^2\to\mbR^2$ and $A\in R^{2\times 2}$ is some deterministic matrix.
Denote by
$$
\varphi(t)=\int_{\mbR^2}v\mu_t(dv).
$$
Note that
$$
\varphi(t)=\int_{\mbR^2}x(v,t)\mu_0(dv)=
$$
$$
=\int_Dx(\eta(u),t)du.
$$
Using the variation of parameters formula one can obtain
\begin{equation}
\label{eq3.2}
x(v,t)=e^{At}v+\int^t_0e^{A(t-\tau)}\varphi(\tau)d\tau.
\end{equation}
It implies that
	$$
	\varphi(t)=e^{(A+I)t}\varphi(0).
	$$
Substituting in \eqref{eq3.2} the expression for $\varphi$ one can obtain
	$$
	x(v,t)=e^{At}v+(e^{t}-1)e^{At}\varphi(0).
	$$
Writing $\eta(u)$ for $v$ in the previous expression one can get the following process
	$$
	x(\eta(u),t)=e^{At}\eta(u)+(e^{t}-1)e^{At}\varphi(0).
	$$
Denote by
	$$
	B(t)=e^{At},\ C(t)=(e^{t}-1)e^{At}.
	$$
Note that
	$$
	\varphi(0)=\int_D \eta(u)du.
	$$
Then
	$$
	x(\eta(u),t)=B(t)\eta(u)+C(t)\int_D \eta(u)du.
	$$
Note that according to Liouville's formula
	$$
	\det B(t)=e^{t\  tr A}.
	$$
It follows from Theorem \ref{thm3.1} that there exists
$$
\cT^{x(\eta,t)}_{k}=L_2-\lim_{\ve\to0}\cT^{x(\eta,t)}_{\ve,k}
$$
and
$$
\cT^{x(\eta,t)}_{k}=e^{-(k-1)t\  tr A}\cT^{\eta}_{k}.
$$
If $tr A>0,$ then $\cT^{x(\eta,t)}_{k}\to 0,\ t\to +\infty$ a.s. If $tr A<0,$ then $|\cT^{x(\eta,t)}_{k}|\to +\infty,\  t\to +\infty$ on the set $\{\omega:\ \cT^{\eta}_{k}\neq 0 \}.$
\end{expl}
\section{Self-intersection local time for Gaussian field in the flow of interacting particles}
Consider deterministic equation with interaction in $\mbR^2$
\begin{equation}
\label{eq4.1}
\begin{cases}
dx(v,t)=a(x(v,t),\mu_t)dt,\\
x(v,0)=v,\ v\in\mbR^2,\\
\mu_t=\mu\circ x(\cdot,t)^{-1}.
\end{cases}
\end{equation}
Suppose that coefficient $a:\mbR^2\times\cM\to\mbR^2$ is two times continuously differentiable with respect to spatial variable. Here $\cM$ is the space of all probability measures on $\cB(\mbR^2)$ with Wasserstein distance $\gamma$, i.e. for any $\nu_1,\nu_2\in\cM$
$$
\gamma(\nu_1,\nu_2)=\inf_{C(\nu_1,\nu_2)}\int_{\mbR^2}\int_{\mbR^2}\frac{\|u-v\|}{1+\|u-v\|}\kappa(du,dv),
$$
where $C(\nu_1,\nu_2)$ is a set of all probabilities measures on $\cB(\mbR^4)$ with marginal distributions $\nu_1$ and $\nu_2.$ Suppose that $a$ and its derivatives satisfy Lipschitz condition with respect to spatial and measure-valued variables with constants $L$ and $L^{\prime}$ and derivatives of $a$ are bounded. Let the initial measure $\mu_0$ in \eqref{eq4.1} be an occupation measure of Gaussian field $\eta:D\to\mbR^2$ defined in Section 1. The main aim of present section is to define a self-intersection local time for a field
$x(\eta(u),t),\ u\in D$ with $t$ being fixed. Here $x$ is solution to \eqref{eq4.1} with initial measure equals an occupation measure of the field $\eta.$ Hence in comparison with Section 3 now diffeomorphism $x$ is random and depends on $\eta.$ These new difficulties can be solved in a few steps. Firstly one can approximate an occupation measure of the field $\eta$ by the sequence of discrete measures $\mu_n,\ n\geq1$ depending on the finite number of values of the field $\eta.$ After that one can consider equation \eqref{eq4.1} with the initial measure $\mu_n.$ Let $x_n$ be a solution to equation \eqref{eq4.1}. The diffeomorphism $x_n$ is random and depends on the finite number of values of the field $\eta.$ Therefore one have to understand how to define self-intersection local time for the field $x_n(\eta(u),t),\ u\in D.$ Secondly, one have to prove the existence of limit of self-intersection local time as $n\to+\infty.$ To make the first step consider $F:\mbR^{2(n+1)}\to\mbR^2.$ Suppose that $F$ is diffeomorphism with respect to last variable under the first $n$ variables being fixed. Construct random diffeomorphism depending on the finite number of values of the field $\eta$ as follows
$$
F(\eta(v_1),\ldots,\eta(v_n),\eta(u)),\ u\in D.
$$
The following statement shows that for defined random field self-intersection local time exists.
\begin{thm}
\label{thm4.1}
There exists self-intersection local time
$$
\int_{D}\ldots\int_{D}\prod^{k-1}_{j=1}\delta_0\Big(F(\eta(v_1),\ldots,\eta(v_n),\eta(u_{j+1}))-
$$
$$
-F(\eta(v_1),\ldots,\eta(v_n),\eta(u_{j}))\Big)d\vec{u}.
$$
\end{thm}
\begin{proof} Similar to arguments of Section 2 it suffices to prove the statement in the case when derivative of $F$ with respect to $n+1$ variable is bounded. So from now we suppose that
there exist positive constants $c_1,\ c_2$ such that
\begin{equation}
\label{eq4.2}
c_1\leq|\det F^{\prime}_{n+1}|\leq c_2.
\end{equation}
Here $F^{\prime}_{n+1}$ is a derivative with respect to $n+1$ variable.
Note that for an arbitrary $N\geq1$ the distribution of
$$
\eta(v_1),\ldots,\eta(v_n),\eta(u_1),\ldots,\eta(u_N)
$$
is nondegenerated. To approximate the self-intersection local time for the random field
$$
F(\eta(v_1),\ldots,\eta(v_n),\eta(u)),\ u\in D
$$
we will use the same delta family of functions which we introduced in Section 3
$$
f^{F}_{\ve}(v_1,\ldots,v_{k})=\frac{1}{|\det F^{\prime}(F^{-1}(v_1))|^{k-1}}\prod^{k-1}_{i=1}f_{\ve}(F^{-1}(v_{i+1})-F^{-1}(v_{i})),
$$
where
$$
f_{\ve}(y)=\frac{1}{2\pi\ve}e^{-\frac{\|y\|^2}{2\ve}},\ y\in\mbR^2.
$$
Set
$$
I_{\ve}=\int_{D}\ldots\int_{D}\frac{1}{|\det F^{\prime}_{n+1}(\eta(v_1),\ldots,\eta(v_n),\eta(u_1))|^{k-1}}\cdot
$$
$$
\cdot\prod^{k-1}_{j=1}f_{\ve}(\eta(u_{i+1})-\eta(u_i))d\vec{u}.
$$
Then to prove the theorem it suffices to check that there exists a finite limit
$EI_{\ve_1}I_{\ve_2}$ as $\ve_1,\ \ve_2\to0.$
Note that
$$
EI_{\ve_1}I_{\ve_2}=
$$
$$
=\int_{D}\ldots\int_{D}E\ \frac{1}{|\det F^{\prime}_{n+1}(\eta(v_1),\ldots,\eta(v_n),\eta(u_1))|^{k-1}}\cdot
$$
$$
\cdot\frac{1}{|\det F^{\prime}_{n+1}(\eta(v_1),\ldots,\eta(v_n),\eta(u^{\prime}_1))|^{k-1}}\cdot
$$
$$
\cdot \prod^{k-1}_{i_{1}=1}f_{\ve_1}(\eta(u_{i_{1}+1})-\eta(u_{i_1}))\cdot
$$
$$
\cdot\prod^{k-1}_{i_{2}=1}f_{\ve_2}(\eta(u^{\prime}_{i_{2}+1})-\eta(u^{\prime}_{i_2})) d\vec{u}d\vec{u^{\prime}}.
$$
Let us take a conditional expectation with respect to $\eta(v_1),\ldots,\eta(v_n).$
For $u\in D,\ \eta(u)$ admits representation
$$
\eta(u)=f(u)+\zeta(u),
$$
where
$$
f(u)=E(\eta(u)|\ \eta(v_1),\ldots,\eta(v_n))
$$
and $\zeta(u),\ u\in D$ is centered $\mbR^2$-valued Gaussian random field. Moreover for any $u\in D$ random vectors $\zeta(u)$ and $f(u)$ are independent. Assumptions on covariance function of $\eta$ imply that $f(u),\ u\in D$
is a continuous function. One can check that
$$
Var(\zeta_1(u_{j+1})|\zeta_1(u_{1}),\ldots,\zeta_1(u_{j}))=
$$
\begin{equation}
\label{eq4.5}
=Var(\eta_1(u_{j+1})|\eta_1(v_1),\ldots,\eta_1(v_n),\eta_1(u_{1}),\ldots,\eta_1(u_{j})).
\end{equation}
Note that
$$
EI_{\ve_1}I_{\ve_2}=EE\Big(\int_{D}\ldots\int_D\ \frac{1}{|\det F^{\prime}_{n+1}(\eta(v_1),\ldots,\eta(v_n),\eta(u_1))|^{k-1}}\cdot
$$
$$
\cdot\frac{1}{|\det F^{\prime}_{n+1}(\eta(v_1),\ldots,\eta(v_n),\eta(u^{\prime}_1))|^{k-1}}
 \prod^{k-1}_{i_{1}=1}f_{\ve_1}(\eta(u_{i_{1}+1})-\eta(u_{i_1}))\cdot
$$
$$
\cdot\prod^{k-1}_{i_{2}=1}f_{\ve_2}(\eta(u^{\prime}_{i_{2}+1})-\eta(u^{\prime}_{i_2})) d\vec{u}d\vec{u^{\prime}}| \eta(v_1),\ldots,\eta(v_n)\Big)=
$$
$$
=E_\eta\ E_\zeta \int_{D}\ldots\int_D\ \frac{1}{|\det F^{\prime}_{n+1}(\eta(v_1),\ldots,\eta(v_n),f(u_1)+\zeta(u_1))|^{k-1}}\cdot
$$
$$
\cdot\frac{1}{|\det F^{\prime}_{n+1}(\eta(v_1),\ldots,\eta(v_n),f(u^{\prime}_1)+\zeta(u^{\prime}_1))|^{k-1}}\cdot
$$
$$
\cdot \prod^{k-1}_{i_{1}=1}f_{\ve_1}(\Delta\zeta(u_{i_{1}})+\Delta f(u_{i_1}))
\prod^{k-1}_{i_{2}=1}f_{\ve_2}(\Delta\zeta(u^{\prime}_{i_{2}})+\Delta f(u^{\prime}_{i_2})) d\vec{u}d\vec{u^{\prime}},
$$
where $E_{\eta}$ and $E_{\zeta}$ are expectations with respect to $(\eta(v_1),\ldots,\eta(v_n))$ and $(\zeta(u_1),\ldots,\zeta(u_k)-\zeta(u_{k-1}),\zeta^{\prime}(u_1),\ldots,\zeta(u^{\prime}_{k})-\zeta(u^{\prime}_{k-1})).$
Using Lebesgue dominated convergence theorem let us check that there exists finite limit
$$
\int_{D}\ldots\int_D E_\zeta\frac{1}{|\det F^{\prime}_{n+1}(\eta(v_1),\ldots,\eta(v_n),f(u_1)+\zeta(u_1))|^{k-1}}\cdot
$$
$$
\cdot\frac{1}{|\det F^{\prime}_{n+1}(\eta(v_1),\ldots,\eta(v_n),f(u^{\prime}_1)+\zeta(u^{\prime}_1))|^{k-1}}\cdot
$$
$$
\cdot \prod^{k-1}_{i_{1}=1}f_{\ve_1}(\Delta\zeta(u_{i_{1}})+\Delta f(u_{i_1}))\prod^{k-1}_{i_{2}=1}f_{\ve_2}(\Delta\zeta(u^{\prime}_{i_{2}+1})+\Delta f(u^{\prime}_{i_2})) d\vec{u}d\vec{u^{\prime}}
$$
as $\ve_1,\ve_2\to0.$ To do this let us notice that
a covariance matrix of the vector
$$
(\zeta(u_1),\ldots,\zeta(u_k)-\zeta(u_{k-1}),\zeta^{\prime}(u_1),\ldots,\zeta(u^{\prime}_{k})-\zeta(u^{\prime}_{k-1}))
$$
is nondegenerated. Hence the density of this vector $\tilde{p}$ can be represented as follows
$$
\tilde{p}(x_1,y_1,\ldots,y_{k-1},x^{\prime}_1,y^{\prime}_1,\ldots,y^{\prime}_{k-1})=
$$
$$
=p(x_1,x^{\prime}_1)q_{x_1,x^{\prime}_1}(y_1,\ldots,y_{k-1},y^{\prime}_1,\ldots,y^{\prime}_{k-1}),
$$
where a conditional density $q_{x_1,x^{\prime}_1}$ is continuous and bounded with respect to all variables. Then one can conclude that
$$
E_\zeta\ \frac{1}{|\det F^{\prime}_{n+1}(\eta(v_1),\ldots,\eta(v_n),f(u_1)+\zeta(u_1))|^{k-1}}\cdot
$$
$$
\cdot\frac{1}{|\det F^{\prime}_{n+1}(\eta(v_1),\ldots,\eta(v_n),f(u^{\prime}_1)+\zeta(u^{\prime}_1))|^{k-1}}\cdot
$$
$$
\cdot \prod^{k-1}_{i_{1}=1}f_{\ve_1}(\Delta\zeta(u_{i_{1}})+\Delta f(u_{i_1}))\prod^{k-1}_{i_{2}=1}f_{\ve_2}(\Delta\zeta(u^{\prime}_{i_{2}})+ \Delta f(u^{\prime}_{i_2})) d\vec{u}d\vec{u^{\prime}}
$$
converges to
$$
\int_{\mbR^2}\int_{\mbR^2}\ \frac{1}{|\det F^{\prime}_{n+1}(\eta(v_1),\ldots,\eta(v_n),f(u_1)+x_1)|^{k-1}}\cdot
$$
$$
\cdot\frac{1}{|\det F^{\prime}_{n+1}(\eta(v_1),\ldots,\eta(v_n),f(u^{\prime}_1)+x^{\prime}_1)|^{k-1}}\cdot
$$
$$
\cdot p(x_1,x^{\prime}_1) q_{x_1,x^{\prime}_1}(\Delta f(u_1),\ldots,\Delta f(u_{k-1}),\Delta f(u^{\prime}_1),\ldots,\Delta f(u^{\prime}_{k-1}))dx_1dx^{\prime}_1
$$
as $\ve_1,\ \ve_2\to0.$
Moreover, using estimates \eqref{eq4.2} one can see that
$$
E_\zeta\ \frac{1}{|\det F^{\prime}_{n+1}(\eta(v_1),\ldots,\eta(v_n),f(u_1)+\zeta(u_1))|^{k-1}}\cdot
$$
$$
\cdot\frac{1}{|\det F^{\prime}_{n+1}(\eta(v_1),\ldots,\eta(v_n),f(u^{\prime}_1)+\zeta(u^{\prime}_1))|^{k-1}}\cdot
$$
$$
\cdot \prod^{k-1}_{i_{1}=1}f_{\ve_1}(\Delta\zeta(u_{i_{1}})+\Delta f(u_{i_1}))
\prod^{k-1}_{i_{2}=1}f_{\ve_2}(\Delta\zeta(u^{\prime}_{i_{2}})+\Delta f(u^{\prime}_{i_2})) \leq
$$
$$
\leq\frac{c}{G(\Delta\zeta(u_1),\ldots,\Delta\zeta(u_{k-1}),\Delta\zeta(u^{\prime}_1),\ldots,\Delta\zeta(u^{\prime}_{k-1}))},\ c>0,
$$
$$
\Delta\zeta(u_i)=\zeta(u_i)-\zeta(u_{i-1}).
$$
Let us check that
$$
\int_{D}\ldots\int_D\frac{1}{G(\Delta\zeta(u_1),\ldots,\Delta\zeta(u_{k-1}),\Delta\zeta(u^{\prime}_1),\ldots,\Delta\zeta(u^{\prime}_{k-1}))}d\vec{u}<+\infty.
$$
Really,
$$
G(\Delta\zeta(u_1),\ldots,\Delta\zeta(u_{k-1}),\Delta\zeta(u^{\prime}_1),\ldots,\Delta\zeta(u^{\prime}_{k-1}))=
$$
$$
= E\Delta\zeta(u_1)^2\  Var(\Delta\zeta(u_2)|\Delta\zeta(u_1))\ldots
 $$
 $$
 \ldots Var(\Delta\zeta(u^{\prime}_{k-1})|\Delta\zeta(u_1),\ldots,\Delta\zeta(u^{\prime}_{k-1}))\geq
$$
$$
\geq Var(\zeta(u_2)|\zeta(u_1))Var(\zeta(u_3)|\zeta(u_1),\zeta(u_2))\ldots
$$
\begin{equation}
\label{eq4.6}
\ldots Var(\zeta(u^{\prime}_{k})|\zeta(u_1),\ldots,\zeta(u^{\prime}_{k-1})).
\end{equation}
Using  relation \eqref{eq4.5} one can conclude that \eqref{eq4.6} equals
$$
Var(\eta(u_2)|\eta(v_1),\ldots,\eta(v_n),\eta(u_1))\cdot
$$
$$
\cdot Var(\eta(u_3)|\eta(v_1),\ldots,\eta(v_n),\eta(u_1),\eta(u_2))\ldots
$$
$$
\ldots Var(\eta(u^{\prime}_{k})|\eta(v_1),\ldots,\eta(v_n),\eta(u_1),\ldots,\eta(u^{\prime}_{k-1})).
$$
It was shown in the proof of Theorem \ref{thm2.1} that
$$
\int_{D}\ldots\int_{D}\Big(Var\Big(\eta(u_2)|\eta(v_1),\ldots,\eta(v_n),\eta(u_1)\Big)\cdot
$$
$$
\cdot Var\Big(\eta(u_3)|\eta(v_1),\ldots,\eta(v_n),\eta(u_1),\eta(u_2)\Big)\ldots
$$
$$
\ldots Var\Big(\eta(u^{\prime}_{k})|\eta(v_1),\ldots,\eta(v_n),\eta(u_1),\ldots,\eta(u^{\prime}_{k-1})\Big)\Big)^{-1}d\vec{u}<\ c,
$$
where constant $c>0$ does not depend on points
$v_1,\ldots,v_n.$
\end{proof}
Let us apply Theorem \ref{thm4.1} to define a self-intersection local time for the field
$$
x(\eta(u),t),\ u\in D,
$$
where $x(v,t),\ v\in\mbR^2,\ t>0$ is a solution to deterministic equation with interaction \eqref{eq4.1} with the initial measure $\mu$ corresponding to occupation measure of Gaussian random field $\eta(u),\ u\in D.$ For formal description of self-intersection local time for the field $x(\eta(u),t),\ u\in D$ we use the following formula
\begin{equation}
\label{eq4.7}
\cT^{x(\eta,t)}_k=\int_D\ldots\int_D\frac{1}{|\det x^{\prime}(\eta(u_1))|}\prod^{k-1}_{j=1}\delta_0(\eta(u_{j+1})-\eta(u_j))d\vec{u}.
\end{equation}
It was shown in Section 3 that if $x(\cdot,t)$ is deterministic diffeomorphism, then expression \eqref{eq4.7} is well defined. Now $x(\cdot,t)$ is random diffeomorphism which depends on the occupation measure $\mu$ of the random field $\eta.$ To prove the existence of \eqref{eq4.7} in this case we will apply Theorem \ref{thm4.1}. In the beginning let us check the following statement related to approximation of an occupation measure by the sequence of discrete measures.
\begin{lem}
\label{lem4.3} Consider the sequence of random measures
\begin{equation}
\label{eq4.8}
\mu_n=\sum^{n}_{k_1,k_2=0}\frac{1}{n^2}\ \delta_{\eta(\frac{k_1}{n},1+\frac{k_2}{n})},\ n\geq1.
\end{equation}
Then
$$
\gamma(\mu_n,\mu)\to0,\ n\to+\infty, \ a.s.
$$
\end{lem}
\begin{proof}
Since $\eta$ has a continuous modification on $D,$ then
$$
\gamma_0(\mu_n,\mu)\leq
$$
$$
\leq \sum^{n-1}_{k_1,k_2=0}\frac{\|\eta(\frac{k_1}{n},1+\frac{k_2}{n})-\eta(u)\|}{1+\|\eta(\frac{k_1}{n},1+\frac{k_2}{n})-\eta(u)\|}1_{[\frac{k_1}{n};\frac{k_1+1}{n}]\times[1+\frac{k_2}{n};1+\frac{k_2+1}{n}]}(u)du\to0
$$
when $n\to+\infty,\ a.s.$
\end{proof}
Let $x_n$ be a solution to equation \eqref{eq4.1} with the initial measure $\mu_n$ defined by formula \eqref{eq4.8}. Denote by
$$
I^n_{\ve}=\int_{D}\ldots\int_D\frac{1}{|\det x^{\prime}_{n}(\eta(u_1),t)|^{k-1}}\prod^{k-1}_{j=1}f_{\ve}(\eta(u_{j+1})-\eta(u_{j}))d\vec{u}.
$$
It follows from Theorem \ref{thm4.1} that for every $n\geq 1$ there exists
$$
I^n=L_2-\lim_{\ve\to0}I^n_{\ve}.
$$
The random variable $I^{n}$ can be treated as the self-intersection local time of multiplicity $k$ for the field $x_n(\eta(u),t),\ u\in D.$  The next theorem states that there exists a limit, almost surely, of $I^n$ which will be the self-intersection local time for the field $x(\eta(u),t),\ u\in D$.
\begin{thm}
\label{thm4.2} There exists
$$
I=\lim_{n\to+\infty}I^n,\ a.s.
$$
\end{thm}
\begin{proof} To prove the theorem let us check that $\{I^n,\ n\geq1\}$ is Cauchy sequence. Since there exists
$$
I^{1}=P-\lim_{\ve\to0}I^{1}_\ve,
$$
then it follows from Riesz lemma that there exist subsequence $\{\ve_{1m},\ m\geq1\}$ such that
$I^{1}_{\ve_{1m}}\to I^{1}$ as $m\to+\infty,\ a.s.$ From existence of
$$
I^{2}=P-\lim_{m\to+\infty}I^{2}_{\ve_{1m}}
$$
one can conclude the existence of a subsequence
$\{\ve_{2m},\ m\geq1\}$ which is subsequence of $\{\ve_{1m},\ m\geq1\}$ such that
$I^{1}_{\ve_{2m}}\to I^{1}$ and
$I^{2}_{\ve_{2m}}\to I^{2}$ as $m\to+\infty,\ a.s.$ Repeating this procedure one can obtain  nested system of subsequences $\Big\{\{\ve_{lm},\ m\geq1\},\ l\geq1\Big\}$ such that
$I^{1}_{\ve_{lm}}\to I^{1},\ I^{2}_{\ve_{lm}}\to I^{2},\ldots,I^{l}_{\ve_{lm}}\to I^{l}$
as $m\to+\infty,\ a.s.$ Taking $\{\ve_{mm},\ m\geq1\}$ we get for any $n,l\geq 1$
$$
|I^{n}-I^{l}|=\lim_{m\to+\infty}|I^{n}_{\ve_{mm}}-I^{l}_{\ve_{mm}}|,\ a.s.
$$
Also we can suppose that there exists
$$
\lim_{m\to+\infty}\int_{D^k}\prod^{k-1}_{j=1}f_{\ve_{mm}}(\eta(u_{j+1})-\eta(u_{j}))d\vec{u},\ a.s.
$$
To proceed further we need the following statement which can be proved using arguments of \cite{16}.
\begin{lem}
\label{lem4.1} Let $x^{\mu},\ x^{\nu}$ be solution to \eqref{eq4.1} corresponding to initial measures $\mu$ and $\nu.$ Under mentioned conditions  the following relation holds
\begin{enumerate}
\item
There exists constants $c_1>0$ such that for any $u\in\mbR^2$ and $t\in[0;1]$
$$
\|x^{\mu}(u,t)-x^{\nu}(u,t)\|\leq c_1\gamma(\mu,\nu)
$$
\item
There exists the derivative of $x$ with respect to spatial variable $x^{\prime}_u(u,t)\in C(\mbR^d\times[0;1],\mbR^{d\times d})$ and
$$
\det x^{\prime}_u(u,t)=e^{\int^t_0 tr\ a^{\prime}(x(u,s),\mu_s)ds}
$$
\item
There exists a constant $c_2>0$ such that for any $u\in\mbR^2$ and $t\in[0;1]$
$$
|tr\ a^{\prime}(x^\mu(u,t),\mu_t)-tr\ a^{\prime}(x^\nu(u,t),\nu_t)|\leq c_2\gamma(\mu,\nu)
$$
\item
There exists a constant $c_3>0$ such that for any $u\in\mbR^2$ and $t\in[0;1]$
$$
\Big|det\ \frac{\partial}{\partial u}x^{\mu}(u,t)-det\ \frac{\partial}{\partial u}x^{\nu}(u,t)\Big|\leq c_3\gamma(\mu,\nu).
$$
\end{enumerate}
\end{lem}
Lemma \ref{lem4.1} describes the dependence of the flow with interaction on initial mass distribution.
It follows from Lemma \ref{lem4.1} that
$$
|I^{n}_{\ve_{mm}}-I^{l}_{\ve_{mm}}|\leq
$$
$$
\leq c\ \gamma(\mu_{n},\mu_{l})\int_{D^k}\prod^{k-1}_{j=1}f_{\ve_{mm}}(\eta(u_{j+1})-\eta(u_{j}))d\vec{u},\ c>0.
$$
Lemma \ref{lem4.3} implies that
$$
\gamma(\mu_{n},\mu_{l})\to0,
$$
as $n,\ l\to+\infty, \ a.s.$
Consequently
$$
\lim_{n,\ l\to+\infty}\lim_{m\to+\infty}|I^{n}_{\ve_{mm}}-I^{l}_{\ve_{mm}}|=0
$$
which proves the theorem.
\end{proof}
\section{Asymptotics of self-intersection local for Gaussian field in stochastic flow}
In Section 4 we defined self intersection local $\cT^{x(\eta),t}_k$ for the Gaussian field $\eta(u),\ u\in D$ evolving in a flow of interacting particles $x(v,t),\ v\in\mbR^2,\ t\geq0.$ In present Section we study asymptotic behavior   of random variable $\cT^{x(\eta),t}_k$ as $t\to+\infty.$

Firstly, consider the field in isotropic Brownian flow on $\mbR^2$ \cite{18}.
Consider a function $\varphi\in C^{\infty}(\mbR^2)$ such that
$$
\varphi(u)=\varphi(\|u\|),\ u\in\mbR^2,\ supp\ \varphi\subset B(0;1),i.e.
$$
$$
\varphi(u)=0,\ \|u\|>1,
$$
and
$$
\int_{\mbR^2}\varphi^2(u)du=1.
$$
For $\ve>0$ put
$$
\varphi_{\ve}(u)=\frac{1}{\ve}\varphi\Big(\frac{u}{\ve}\Big).
$$
Consider the following Cauchy problem
\begin{equation}
\label{eq5.1}
\begin{cases}
dx_{\ve}(u,t)=\int_{\mbR^2}\varphi_{\ve}(x_{\ve}(u,t)-p)W(dp,dt)\\	
x_{\ve}(u,0)=u,\ u\in\mbR^2,
\end{cases}
\end{equation}
where $W$ is  two-dimensional Wiener sheet on $\mbR^2\times [0;+\infty).$ The following statements describe some properties of solution to \eqref{eq5.1}.
\begin{lem}
\label{lem5.1} There exists a unique solution to Cauchy problem \eqref{eq5.1} which is a flow of diffeomorphisms. Moreover
$x_{\ve}$ defines an isotropic Brownian flow on $\mbR^2.$
\end{lem}
The proof of lemma easily follows from properties of function $\varphi.$
\begin{lem}
\label{lem5.3}
Let $x_{\ve}(u,t),\ u\in\mbR^2,\ t\in[0;1]$ be a solution to Cauchy problem \eqref{eq5.1}. Then for $u_1\neq u_2$ random process
$\Big(x_{\ve}(u_1,\cdot),x_{\ve}(u_2,\cdot)\Big)$
converges weakly in $C([0;1],\mbR^4)$ to a process $\Big(u_1+w_1(\cdot),u_2+w_2(\cdot)\Big)$
as $\ve\to0.$ Here $w_1, w_1$  are two independent planar Wiener processes.
\end{lem}
\begin{proof} For fixed $u_1,u_2$ processes $x_{\ve}(u_1,t), x_{\ve}(u_2,t),\ t\in[0;1]$ are planar Wiener processes. Consequently the family
$$
\Big\{\Big(x_{\ve}(u_1,\cdot),x_{\ve}(u_2,\cdot)\Big),\ \ve>0\Big\}
$$
is relatively compact in $C([0;1],\mbR^4).$
Denote by
$$
\Big(x(u_1,\cdot),x(u_2,\cdot)\Big)
$$
some limit point as $\ve\to0$. To finish the proof we need
\begin{lem}
\label{lem5.4} Let $w(t)$ be planar Wiener process starting from the point $(1,0).$ Set
$$
\tau^w_{\ve}=\inf\{t>0:\ \|w(t)\|\leq\ve\}.
$$
Then, almost surely,
$$
\tau^w_{\ve}\to+\infty,\ \ve\to0.
$$
\end{lem}
\begin{proof}
Since for $\ve_2<\ve_1$
$$
\tau^w_{\ve_1}<\tau^w_{\ve_2},
$$
then, almost surely, exists
$$
\tau^w_{0}=\lim_{\ve\to0}\tau^w_{\ve}
$$
finite or infinite. Consider such $\omega$ that
$$
\tau^w_{0}(\omega)<+\infty.
$$
Due to the continuity of $w$
$$
\lim_{\ve\to0}w(\tau^w_{\ve}(\omega))=w(\tau^w_{0}(\omega)).
$$
Consequently
$$
\|w(\tau^w_{0}(\omega))\|=0.
$$
Since
$$
P\{\exists\  t>0:\ \|w(t)\|=0\}=0,
$$
then
$$
P\{\tau_0<+\infty\}=0.
$$
\end{proof}
Let us return to the proof of Lemma \ref{lem5.3}. Denote by
$$
\tau_{\ve}=\inf\{t>0:\|x_{\ve}(u_2,t)-x_{\ve}(u_1,t)\|\geq 2\ve\}.
$$
Consider the process
$$
y(t)=
\begin{cases}
x_{\ve}(u_1,t),\ t\leq\tau_{\ve}\\	
x_{\ve}(u_1,\tau_{\ve})+\tilde{w}(t)-\tilde{w}(\tau_{\ve}),\ t>\tau_{\ve},
\end{cases}
$$
where $\tilde{w}$ is planar Wiener such that $\tilde{w}$ and $x_{\ve}(u_i,\cdot),\ i=1,2$ are independent.
Note that processes $y(\cdot)$ and $x_{\ve}(u_2,\cdot)$ are martingales with respect to filtration
$\cF_t=\sigma\Big(y(r),x_{\ve}(u_2,r),\ r\leq t\Big).$
One can check that processes $x_{\ve}(u_2,\cdot)$ and $y(\cdot)$ are two independent Wiener processes. Let us construct another
planar Wiener process
$$
z(t)=\frac{u_2-u_1}{\sqrt{2}}+\frac{1}{\sqrt{2}}\Big(y(t)-x_\ve(u_2,t)\Big),\ t\geq0.
$$
Note that
$$
\tau^z_{\frac{2\ve}{\sqrt{2}}}=^d\tau_{\ve}.
$$
It follows from Lemma \ref{lem5.4} that for any $t>0$
\begin{equation}
\label{eq5.2}
P\{\tau_{\ve}<t\}\to0,\ \ve\to0.
\end{equation}
Since $supp\  \varphi_{\ve}\ast\varphi_{\ve}\subset B(0,2\ve)$, then using \eqref{eq5.2} one can check that the joint characteristics
$$
\langle x_{\ve}(u_1,\cdot),x_{\ve}(u_2,\cdot)\rangle (t)=
$$
$$
=\int^t_0\varphi_{\ve}\ast\varphi_{\ve}\Big(\|x_{\ve}(u_2,s)-x_{\ve}(u_2,s)\|\Big)ds\to0,\ \ve\to0.
$$
Consequently the  process $(x(u_1,\cdot),(x(u_2,\cdot)))$ is continuous square integrable martingale with
$$
d\langle x(u_i,\cdot),(x(u_j,\cdot)\rangle(t)=\delta_{ij}dt,
$$
where
$$
\delta_{ij}=
\begin{cases}
1,\ i=j\\
0,\ i\neq j.
\end{cases}
$$
Now the statement of Lemma \ref{lem5.3} follows from multidimensional analog of Levy's theorem \cite{22} (p. 352, Th. 18.3).
\end{proof}
Denote by $x^{\prime}_{\ve}(u,t)$ the derivative with respect to spatial variable. Then it can be checked that the following statement holds.
\begin{lem}
\label{lem5.2}
$$
det\  x^{\prime}_{\ve}(u,t)=exp\Big\{\frac{1}{\ve^2}\int^t_0\int_{\mbR^2}(\nabla \varphi(\frac{x(u,t)-p}{\ve}),W(dp,ds))-\frac{ct}{2\ve^2}\Big\},
$$
where
$$
c=\int_{\mbR^2}\|\nabla\varphi(q)\|^2dq.
$$
\end{lem}
\begin{proof}
 Note that \eqref{eq5.1} can be represented as follows
\begin{equation}
\label{eq5.3}
\begin{cases}
dx^1_{\ve}(u,t)=\int_{\mbR^2}\varphi_{\ve}(x_{\ve}(u,t)-p)W_1(dp,dt)\\
dx^2_{\ve}(u,t)=\int_{\mbR^2}\varphi_{\ve}(x_{\ve}(u,t)-p)W_2(dp,dt)\\	
x^1_{\ve}(u,0)=u_1,\ x^2_{\ve}(u,0)=u_2,\ u_1,u_2\in\mbR.
\end{cases}
\end{equation}
Let $\{e_k,\ k\geq1\}$ be an orthonormal basis in $L_2(\mbR^2).$ Then
$$
\varphi_{\ve}(x_{\ve}(u,t)-p)=\sum^{\infty}_{k=1}(\varphi_{\ve}(x_{\ve}(u,t)-\cdot),e_k)e_k(p).
$$
Denote by
$$
\beta^i_k(t)=\int^t_0\int_{\mbR^2}e_k(p)W_i(dp,ds),\ i=1,2.
$$
Note that $\{\beta^i_k(t),\ t\geq0\}$ are independent Brownian motions in $\mbR.$
Set
$$
a_k(x_{\ve}(u,t))=(\varphi_{\ve}(x_{\ve}(u,t)-\cdot),e_k).
$$
Then \eqref{eq5.3} can be rewritten as
\begin{equation}
\label{eq5.4}
\begin{cases}
dx^1_{\ve}(u,t)=\sum^{\infty}_{k=1}a_k(x_{\ve}(u,t))d\beta^1_k(t)\\
dx^2_{\ve}(u,t)=\sum^{\infty}_{k=1}a_k(x_{\ve}(u,t))d\beta^2_k(t)\\	
x^1_{\ve}(u,0)=u_1,\
x^2_{\ve}(u,0)=u_2,\ u_1,u_2\in\mbR.
\end{cases}
\end{equation}
It follows from \eqref{eq5.4} that $x^{\prime}_{\ve}(u,t),\ u\in\mbR^2,\ t\geq0$ satisfies the following stochastic differential equation
\begin{equation}
\label{eq5.5}
\begin{cases}
d\ \frac{\partial x^1_{\ve}(u,t)}{\partial u_1}=\frac{1}{\ve}\sum^{\infty}_{k=1}\frac{\partial {a_k(x_{\ve}(u,t))}}{{\partial x_1}}\frac{\partial x_1(u,t)}{\partial u_1}d\beta^1_k(t)\\
d\ \frac{\partial x^1_{\ve}(u,t)}{\partial u_2}=\frac{1}{\ve}\sum^{\infty}_{k=1}\frac{\partial {a_k(x_{\ve}(u,t))}}{{\partial x_1}}\frac{\partial x_1(u,t)}{\partial u_2}d\beta^1_k(t)\\
d\ \frac{ \partial x^2_{\ve}(u,t)}{\partial u_1}=\frac{1}{\ve}\sum^{\infty}_{k=1}\frac{\partial a_k(x_{\ve}(u,t))}{\partial x_2}\frac{\partial x_2}{\partial u_1}d\beta^2_k(t)\\	
d\ \frac{ \partial x^2_{\ve}(u,t)}{\partial u_2}=\frac{1}{\ve}\sum^{\infty}_{k=1}\frac{\partial a_k(x_{\ve}(u,t))}{\partial x_2}\frac{\partial x_2}{\partial u_2}d\beta^2_k(t)\\	
\frac{\partial x^1_{\ve}(u,0)}{\partial u_1}=1\\
\frac{\partial x^1_{\ve}(u,0)}{\partial u_2}=0\\
\frac{\partial x^2_{\ve}(u,0)}{\partial u_1}=0\\
\frac{\partial x^2_{\ve}(u,0)}{\partial u_2}=1.
\end{cases}
\end{equation}
Put
$$
\alpha^i_k=\frac{\partial {a_k(x_{\ve}(u,t))}}{\partial x_i}.
$$
Applying Ito formula one can check that solution to \eqref{eq5.5} has the following representation
\begin{equation}
\label{eq5.6}
\begin{cases}
\frac{\partial x^1_{\ve}(u,t)}{\partial u_1}=exp\Big\{\frac{1}{\ve}\int^t_0\sum^{\infty}_{k=1}\alpha^1_k(s)d\beta^1_k(s)-\frac{1}{2\ve^2}\int^t_0\sum^{\infty}_{k=1}(\alpha^1_k(s))^2ds\Big\}\\
\frac{\partial x^1_{\ve}(u,t)}{\partial u_2}=0\\
\frac{\partial x^2_{\ve}(u,t)}{\partial u_1}=0\\
\frac{\partial x^2_{\ve}(u,t)}{\partial u_2}=exp\Big\{\frac{1}{\ve}\int^t_0\sum^{\infty}_{k=1}\alpha^2_k(s)d\beta^2_k(s)-\frac{1}{2\ve^2}\int^t_0\sum^{\infty}_{k=1}(\alpha^2_k(s))^2ds\Big\}.
\end{cases}
\end{equation}
It follows from \eqref{eq5.6} that
$$
det\ x^{\prime}(u,t)=exp\Big\{\sum^{\infty}_{k=1}\frac{1}{\ve}\int^t_0\alpha^1_k(s)d\beta^1_k(s)+
$$
 \begin{equation}
\label{eq5.7}
+\frac{1}{\ve}\int^t_0\alpha^2_k(s)d\beta^2_k(s)-\frac{1}{2\ve^2}\sum^{\infty}_{k=1}\int^t_0((\alpha^1_k(s))^2+(\alpha^2_k(s))^2)ds\Big\}.
 \end{equation}
Using \eqref{eq5.7} one can conclude that
$$
det\  x^{\prime}_{\ve}(u,t)=\exp\Big\{\frac{1}{\ve}\int^t_0\int_{\mbR^2}(\nabla \varphi_{\ve}(x_\ve(u,s)-p),W(dp,ds))-
$$
$$
-\frac{1}{2\ve^2}\int^t_0\int_{\mbR^2}\|\nabla\varphi_{\ve}(x_\ve(u,s)-p)\|^2dpds\Big\}=
$$
$$
=\exp\Big\{\frac{1}{\ve^2}\int^t_0\int_{\mbR^2}(\nabla \varphi(\frac{x_\ve(u,s)-p}{\ve}),W(dp,ds))-
$$
$$
-\frac{1}{2\ve^4}\int^t_0\int_{\mbR^2}\|\nabla\varphi(\frac{x_\ve(u,s)-p}{\ve})\|^2dpds\Big\}.
$$
A change of variable
$$
\frac{x_\ve(u,t)-p}{\ve}=q
$$
completes the proof.
\end{proof}
Note that $\varphi_{\ve}(u)=0,$ when $\|u\|>\ve.$ Follow to \cite{21} a parameter $\ve$ can be treated as a radius of interaction of the particles $x(v,t)$ with the molecules of random media $W$. The speed of molecules of the media is much greater then the speed of particles in a flow. So it is reasonable to expect that when $\ve\to0$ the particles in the flow begin to move independently. A rigorous mathematical proof of this fact was given in Lemma \ref{lem5.3}, where one can see that processes $x_{\ve}(u_1,\cdot)$ and $x_{\ve}(u_2,\cdot)$ became independent in the limit.  In particular, it causes the growth of the self-intersection local times for images of random field $\eta.$
Assume that random field $\eta$ and $W$ are independent. Then one can suppose that our probability space is represented as
$$
\Big(\Omega_1\times\Omega_2,\cF_1\otimes\cF_2,P_1\times P_2\Big).
$$
 Suppose that $W$ is defined on $\Omega_1$ and $\eta$ is defined on $\Omega_2$. The solution to Cauchy problem \eqref{eq5.1} has continuous modification with respect to $t.$ Moreover for every $\omega\in \Omega_1$ and $t>0$ it is a diffeomorphism with respect to $u$. For fixed $k\geq1$ consider weight function
$$
\rho(u)=\frac{1}{|\det\ x^{\prime}_{\ve}(u,t)|^{k-1}}.
$$
It follows from Lemma \ref{lem5.2} that
$$
\rho(u)=\exp\Big\{-\frac{k-1}{\ve^2}\int^t_0\int_{\mbR^2}(\nabla \varphi(\frac{x_\ve(u,t)-p}{\ve}),W(dp,ds))+\frac{c(k-1)t}{2\ve^2}\Big\},
$$
$u\in\mbR^2.$ For any $\omega_1\in\Omega_1 $ the function $\rho$ is continuous, since for any $\omega_1\in\Omega_1 $ the derivative of $x_{\ve}$ with respect to $u$ is continuous function.
Then due to Theorem \ref{thm2.1} and Remark \ref{remk2.1} for any $\omega\in \Omega_1$ a self-intersection local time for $x(\eta(u),t),\ u\in D$ with $t\geq0$ being fixed  is well-defined and
$$
\cT^{x_{\ve}(\eta,t)}_k=\int_{D}\ldots\int_{D}\rho(\eta(u_1))\prod^{k-1}_{j=1}\delta_0(\eta(u_{j+1})-\eta(u_{j}))d\vec{u}.
$$
Let us check that the following statement holds.
\begin{lem}
\label{lem5.5} For any $t>0$
$$
\lim_{\ve\to0}e^{-\frac{ck(k-1)t}{2\ve^2}}E\cT^{x_{\ve}(\eta,t)}_k=E\cT^{\eta}_k.
$$
\end{lem}
\begin{proof} Since $\eta$ and $\rho$ are independent, then using approximating families for $\cT^{x_{\ve}(\eta,t)}_k$ and passing to the limit one can deduce the following relations
$$
E\cT^{x_{\ve}(\eta,t)}_k=E\ E\Big(\cT^{x_{\ve}(\eta,t)}_k|\eta\Big)=
$$
$$
=E\ E\Big(\int_{D}\ldots\int_{D}\rho(\eta(u_1))\prod^{k-1}_{j=1}\delta_0(\eta(u_{j+1})-\eta(u_{j}))d\vec{u}|\eta\Big)=
$$
$$
=E^\eta\int_{D}\ldots\int_{D}E^\rho\rho(\eta(u_1))\prod^{k-1}_{j=1}\delta_0(\eta(u_{j+1})-\eta(u_{j}))d\vec{u},
$$
where $E^\eta$ and $E^\rho$ are expectation with respect to $\eta$ and $\rho.$
Denote by
$$
\cE_{\ve}(u,t)=
$$
$$
=\exp\Big\{\Big(-\frac{(k-1)}{\ve^2}\int^t_0\int_{\mbR^2}(\nabla \varphi(\frac{x_\ve(u,t)-p}{\ve}),W(dp,ds))-\frac{c(k-1)^2t}{2\ve^2}\Big)\Big\}.
$$
Set
$$
\cE(u,t)=\cE_1(u,t).
$$
Since  $\cE_\ve(u,t)$ is a stochastic exponential, then
\begin{equation}
\label{eq5.8}
E\cE_\ve(u,t)=1.
\end{equation}
It follows from \eqref{eq5.8} that
$$
E^\rho\rho(\eta(u))=e^{\frac{ck(k-1)t}{2\ve^2}}
$$
which finishes the proof of the lemma.
\end{proof}
Next statement describes a behavior of the image $x_{\ve}(\eta,t).$
\begin{lem}
\label{lem5.6}
$$
\lim_{\ve\to0}E\int_{D}\int_{D}\|x_{\ve}(\eta(u_2),t)-x_{\ve}(\eta(u_1),t)\|^2du_1du_2=
$$
$$
=4t+E\int_{D}\int_{D}\|\eta(u_2)-\eta(u_1)\|^2du_1du_2.
$$
\end{lem}
\begin{proof} One can check that for any $\ve>0$
$$
\|x_{\ve}(\eta(u_2),t)-x_{\ve}(\eta(u_1),t)\|^2\leq
$$
$$
\leq 2E\|\eta(u_2)-\eta(u_1)\|+8t.
$$
Applying Lebesgue's dominated convergence theorem and Lemma \ref{lem5.3} one can conclude that
$$
\lim_{\ve\to0}E\int_{D}\int_{D}\|x_{\ve}(\eta(u_2),t)-x_{\ve}(\eta(u_1),t)\|^2du_1du_2=
$$
$$
=\int_{D}\int_{D}E\|\eta(u_2)-\eta(u_1)+w_2(t)-w_1(t)\|^2du_1du_2=
$$
$$
=4t+E\int_{D}\int_{D}\|\eta(u_2)-\eta(u_1)\|^2du_1du_2,
$$
where $w_1,w_1$ are two independent planar Wiener processes such that $w_1,w_2$ and $\eta$ are independent.
\end{proof}
Together Lemma \ref{eq5.5} and Lemma \ref{eq5.6} describe the interesting features of our Gaussian field in stochastic flow. On the one hand it follows from Lemma \ref{eq5.5} that the expectation of self-intersection local time for the field $x_{\ve}(\eta(u),t),\ u\in D$ in $e^{\frac{ck(k-1)t}{2\ve^2}}$ times greater then the expectation of self-intersection local time for the initial field $\eta(u),\ u\in D$ and grows to infinity as radius of interaction tends to zero. On the other hand one can see from Lemma \ref{eq5.6} that with the decreasing of radius of interaction with the media the value
$$
E\int_D\int_D\|x_{\ve}(\eta(u_2),t)-x_{\ve}(\eta(u_1),t)\|^2du_1du_2
$$
almost does not differ from
$$
E\int_D\int_D\|\eta(u_2)-\eta(u_1)\|^2du_1du_2.
$$

Now let us consider a long-time behavior of the self-intersection local times for the evolving random fields. Lemma \ref{lem5.2} allows to discuss a transformation of random field $\eta$ by the isotropic stochastic flow \eqref{eq5.1}. Consider Cauchy problem
$$
\begin{cases}
dx(u,t)=\int_{\mbR^2}\varphi(x(u,t)-p)W(dp,dt)\\	
x(u,0)=u,\ u\in\mbR^2.
\end{cases}
$$
It follows from Lemma \ref{lem5.2} that
$$
det\  x^{\prime}(u,t)=\exp\Big\{\int^t_0\int_{\mbR^2}(\nabla \varphi(x(u,t)-p),W(dp,ds))-\frac{ct}{2}\Big\},
$$
where $c=\int_{\mbR^2}\|\nabla \varphi(q)\|^2dq$.
In this case the weight function $\rho$ for $k$-multiple self-intersection local time of field $x(\eta(u),t),\ u\in D$ has representation
\begin{equation}
\label{eq5.9}
\rho(u)=exp\Big\{-(k-1)\int^t_0\int_{\mbR^2}(\nabla \varphi(x(u,t)-p),W(dp,ds))+\frac{(k-1)ct}{2}\Big\}.
\end{equation}
One can check that the following statement holds.
\begin{lem}
\label{lem5.7}
$$
\lim_{t\to+\infty}e^{-\frac{ck(k-1)t}{2}}E\cT^{x(\eta,t)}_k=E\cT^{\eta}_k.
$$
\end{lem}
\begin{proof}Using the same arguments as in the proof of Lemma \ref{eq5.5} one can see that
$$
E\cT^{x(\eta,t)}_k=EE(\cT^{x(\eta,t)}_k|\eta)=
$$
$$
=E^\eta\int_D\ldots\int_DE^\rho\rho(\eta(u_1))\prod^{k-1}_{j=1}\delta_0(\eta(u_{j+1})-\eta(u_{j}))d\vec{u}.
$$
Using \eqref{eq5.9} one can conclude that
$$
E^\rho\rho(\eta(u_1))=e^{\frac{ck(k-1)t}{2}}
$$
which finishes the proof of lemma.
\end{proof}
As one can see from Lemma \ref{lem5.7} the expectation of self-intersection local time for the field $x(\eta(u),t),\ u\in D$ grows exponentially as $t\to+\infty.$ To compensate this phenomena let us include in the picture the interaction between particles. Next equation will be equation with interaction of the form
\begin{equation}
\label{eq5.10}
\begin{cases}
dx(u,t)=a(x(u,t)-\int_{\mbR^2}v\mu_t(dv))dt+\int_{\mbR^2}\varphi(x(u,t)-p)W(dp,dt)\\	 x(u,0)=u,\ u\in\mbR^2\\
\mu_t=\mu_0\circ x(\cdot,t)^{-1},\ a>0.
\end{cases}
\end{equation}
In this equation there are two different ways of mass transportation. The first term describes the motion of the particles with respect to joint center of  mass
$$
m_t=\int_{\mbR^2}v\mu_t(dv).
$$
The second term is related to the isotropic stochastic flow. It can be seen from Lemma \ref{lem5.7} that the self-intersection local times in such flow are growing to the infinity as $t$ tends to infinity with the radius of interaction being fixed but we can compensate this growth by the choice of an appropriate $a$ in the first term.
Let $\mu_0$ be an occupation measure of the field $\eta.$ Then equation \eqref{eq5.10} can be rewritten as follows
\begin{equation}
\label{eq5.11}
\begin{cases}
dx(u,t)=a(x(u,t)-\int_{\mbR^2}x(v,t)\mu_0(dv))dt+\int_{\mbR^2}\varphi(x(u,t)-p)W(dp,dt)\\	 x(u,0)=u,\ u\in\mbR^2\\
\mu_t=\mu_0\circ x(\cdot,t)^{-1}.
\end{cases}
\end{equation}
 Using the same arguments as in the proof of Lemma \ref{lem5.2} it is not difficult to conclude that
\begin{equation}
\label{eq5.12}
det\  x^{\prime}(u,t)=\exp\Big\{\int^t_0\int_{\mbR^2}(\nabla \varphi(x(u,t)-p),W(dp,ds))-\frac{ct}{2}+at\Big\},
\end{equation}
where
$$
c=\int_{\mbR^2}\|\nabla\varphi(q)\|^2dq.
$$
Consequently the weight function right now has the representation
$$
\rho(u)=\exp\Big\{-(k-1)\int^t_0\int_{\mbR^2}(\nabla \varphi(x(u,t)-p),W(dp,ds))+
$$
\begin{equation}
\label{eq5.13}
+\frac{c(k-1)t}{2}-a(k-1)t\Big\}.
\end{equation}
Essential difference from the previous cases is that now $x$ depends on $\eta.$ Repeating arguments of Section 4 one can prove that the self-intersection local time $\cT^{x(\eta,t)}_k$ exists and
$$
\cT^{x(\eta,t)}_k=\cT^{\eta}_k(\rho)
$$
with $\rho$ from \eqref{eq5.13}. Moreover repeating arguments of proofs of Lemma \ref{eq5.5} and Lemma \ref{lem5.7} one can see that the following statement holds.
\begin{thm}
\label{thm5.1}
$$
\lim_{t\to+\infty}e^{-(kc-2a)(k-1)t}E\cT^{x(\eta,t)}_k=E\cT^\eta_k.
$$
\end{thm}
Consider random measure on $\mbR^2$ defined as follows for any $A\in\cB(\mbR^2)$
$$
\nu_k(A)=\int_D\ldots\int_D1_{A}(\eta(u_1))\prod^{k-1}_{i=1}\delta_0(\eta(u_{i+1})-\eta(u_{i}))d\vec{u}.
$$
(for proof of existence of such measure see Lemma \ref{lem2.6}).
Let $x$ be solution to equation with interaction \eqref{eq5.11}. Then $\cT^{x(\eta,t)}_k$ admits the following representation
$$
\cT^{x(\eta,t)}_k=\int_{\mbR^2}\rho(u)\nu_k(du),
$$
where the weight-function $\rho$ is defined in \eqref{eq5.13}. Denote by
$$
\beta(u,t)=\int^t_0\int_{\mbR^2}(\nabla \varphi(x(u,t)-p),W(dp,ds)).
$$
Denote by
$$
\cF_t=\sigma(\eta(u),\ u\in D,\ W(\Delta),\ \Delta\subset [0;t]\times\mbR^2).
$$
It can be checked that the random process $\beta(u,t),\ t\in[0;1]$ is continuous square integrable martingale with respect to $\cF_t$ with the following quadratic characteristics
$$
\langle\beta(u,\cdot)\rangle(t)=ct,
$$
where
$$
c=\int_{\mbR^2}\|\nabla\varphi(q)\|^2dq.
$$
Set
$$
m_k(t)=\exp\Big\{a(k-1)t-\frac{c(k-1)t}{2}-\frac{c(k-1)^2t}{2}\Big\}\cT^{x(\eta,t)}_k=
$$
$$
=\int_{\mbR^2}\exp\Big\{-(k-1)\beta(u,t)-\frac{c(k-1)^2t}{2}\Big\}\nu_k(du).
$$
Put
$$
\cE(u,t)=\exp\Big\{-(k-1)\beta(u,t)-\frac{c(k-1)^2t}{2}\Big\}.
$$

\begin{lem}
\label{lem5.8}The random process
$m_k(t)\ t\in[0;1]$ is a positive continuous square integrable martingale with respect to $\cF_t$ with the following quadratic characteristics
$$
\langle m_k\rangle(t)=(k-1)^2\int_{\mbR^2}\int_{\mbR^2}\int^t_0\cE(u,s)\cE(v,s)\cdot
$$
$$
\cdot(\nabla \varphi\ast\nabla\varphi(x(u,s)-x(v,s)),ds)\nu_k(du)\nu_k(dv).
$$
\end{lem}
\begin{proof}Since $\beta(u,t), \ t\in[0;1]$ is a continuous square integrable martingale with respect to filtration $\cF_t$, then it follows from \cite{23} that $\cE(u,t)$ is also continuous square integrable martingale with respect to filtration $\cF_t$
which satisfies the following SDE
$$
\begin{cases}
d\cE(u,t)=\cE(u,t)d\beta^k(u,t)\\
\cE(u,0)=1,
\end{cases}
$$
where
$$
\beta^k(u,t)=-(k-1)\beta(u,t).
$$
 A joint characteristics of martingales $\cE(u,t)$ and $\cE(v,t)$ has a representation \cite{23} (II, Prop. 2.3, p. 56)
$$
\langle\cE(u,\cdot)\cE(v,\cdot)\rangle(t)=
$$
$$
=\int^t_0\cE(u,s)\cE(v,s)d\langle\beta^k(u,\cdot),\beta^k(v,\cdot)\rangle(s).
$$
Let us find a joint characteristics of martingales $\beta^k(u,\cdot)$ and $\beta^k(v,\cdot).$
Since
$$
\beta^k(u,t)=-(k-1)\sum^{\infty}_{l=1}\int^t_0(\alpha_l(x(u,s)),d\beta_l(s)),
$$
then
$$
\langle\beta^k(u,\cdot),\beta^k(v,\cdot)\rangle(t)=
$$
$$
=(k-1)^2\sum^{\infty}_{l=1}\int^t_0(\alpha_l(x(u,s),\alpha_l(x(v,s)ds=
$$
$$
=(k-1)^2\int^t_0\int_{\mbR^2}(\nabla \varphi(x(u,s)-p),\nabla \varphi(x(v,s)-p)))dpds.
$$
Consequently
$$
\langle\cE(u,\cdot)\cE(v,\cdot)\rangle(t)=
$$
$$
=(k-1)^2\int^t_0\int_{\mbR^2}\cE(u,s)\cE(v,s)\cdot
$$
\begin{equation}
\label{eq5.15}
\cdot(\nabla \varphi(x(u,s)-p),\nabla \varphi(x(v,s)-p))dpds.
\end{equation}
It is obvious that $m_k$ is positive. Let us check that $m_k$ is martingale. Note that for $s<t$
\begin{equation}
\label{eq5.20}
E(m_k(t)|\cF_s)=E(\int_{\mbR^2}\cE(u,t)\nu_k(du)|\cF_s).
\end{equation}
The random measure $\nu_k$ is measurable with respect to $\cF_0\subset\cF_s$.
Consequently \eqref{eq5.20} equals
$$
\int_{\mbR^2}E(\cE(u,t)|\cF_s)\nu_k(du)=\int_{\mbR^2}\cE(u,s)\nu_k(du).
$$
To prove that $m_k(t),\ t\in[0;1]$ is square integrable martingale, let us estimate
$$
E(m_k(t))^2=E\int_{\mbR^2}\int_{\mbR^2}\cE(u_1,t)\cE(u_2,t)\nu_k(du_1)\nu_k(du_2)=
$$
$$
=EE(\int_{\mbR^2}\int_{\mbR^2}\cE(u_1,t)\cE(u_2,t)\nu_k(du_1)\nu_k(du_2)|\cF_0)=
$$
\begin{equation}
\label{eq5.21}
E^{\eta}\int_{\mbR^2}\int_{\mbR^2}E^{W}\cE(u_1,t)\cE(u_2,t)\nu_k(du_1)\nu_k(du_2),
\end{equation}
where by $E^\eta$ and $E^{W}$ we denoted expectations with respect to $\eta$ and $W.$ Applying Cauchy's inequality one can conclude that \eqref{eq5.21} less or equal to
$$
E^{\eta}\int_{\mbR^2}\int_{\mbR^2}\sqrt{E^{W}\cE(u_1,t)^2}\sqrt{E^{W}\cE(u_2,t)^2}\nu_k(du_1)\nu_k(du_2).
$$
Since
$$
E\cE(u,t)=1,
$$
then
$$
E\cE(u,t)^2=E\exp\Big\{-2(k-1)\beta(u,t)-\frac{2c(k-1)^2t}{2}\Big\}=
$$
$$
=E\exp\Big\{-2(k-1)\beta(u,t)-\frac{4c(k-1)^2t}{2}+\frac{2c(k-1)^2t}{2}\Big\}=
$$
$$
=e^{\frac{2c(k-1)^2t}{2}}\leq e^{\frac{2c(k-1)^2}{2}}
$$
which completes the proof that $m_k(t),\ t\in[0;1]$ is square integrable martingale. To check the continuity let us apply Kolmogorov's continuity theorem. Note that
\begin{equation}
\label{eq5.22}
E(m_k(t_2)-m_k(t_1))^4=E\Big(\int_{\mbR^2}(\cE(u,t_2)-\cE(u,t_1))\nu_k(du)\Big)^4.
\end{equation}
Applying Jensen's inequality and taking conditional expectation with respect to $\cF_0$  one can conclude that \eqref{eq5.22} less or equal to
$$
E\nu_k(\mbR^2)^3\int_{\mbR^2}(\cE(u,t_2)-\cE(u,t_1))^4\nu_k(du)=
$$
\begin{equation}
\label{eq5.23}
=EE\Big(\nu_k(\mbR^2)^3\int_{\mbR^2}\cE(u,t_1)^4\Big(\frac{\cE(u,t_2)}{\cE(u,t_1)}-1\Big)^4\nu_k(du)|\cF_0\Big).
\end{equation}
Denote by
$$
\cE(u,t_1,t_2)=\frac{\cE(u,t_2)}{\cE(u,t_1)}=
$$
$$
=\exp\Big\{-(k-1)(\beta(u,t_2)-\beta(u,t_1))-\frac{c(k-1)^2(t_2-t_1)}{2}\Big\}.
$$
Note that $\cE(u,t_1,t_2)$ is a stochastic exponential. Then \eqref{eq5.23} equals
$$
E^\eta\nu(\mbR^2)^3\int_{\mbR^2}E^{W}\cE(u,t_1)^4(\cE(u,t_1,t_2)-1)^4\nu_k(du).
$$
Due to Cauchy inequality
$$
E^{W}\cE(u,t_1)^4(\cE(u,t_1,t_2)-1)^4\leq(E^{W}\cE(u,t_1)^8)^{\frac{1}{2}}(E^{W}(\cE(u,t_1,t_2)-1)^8)^{\frac{1}{2}}\leq
$$
$$
\leq c_1(E^W(\cE(u,t_1,t_2)-1)^8)^{\frac{1}{2}}.
$$
Since
$$
\cE(u,t_1,t_2)=1+\int^{t_2}_{t_1}\cE(u,s)d\beta^k(u,s),
$$
then it follows from Burkholder-Davis-Gundy inequality that
$$
E^W(\cE(u,t_1,t_2)-1)^8=E^W\Big(\int^{t_2}_{t_1}\cE(u,s)d\beta^k(u,s)\Big)^8\leq
$$
$$
\leq c_2E^W\Big(\int^{t_2}_{t_1}\cE(u,s)^2ds\Big)^4\leq
$$
$$
\leq c_2(t_2-t_1)^3\int^{t_2}_{t_1}E^W\cE(u,s)^8ds\leq
$$
$$
\leq c_3(t_2-t_1)^4.
$$
Due to Lemma \ref{lem2.7}  $E^\eta(\mbR^2)^3<+\infty.$ Consequently the process $m_k$ satisfies Kolmogorov's continuity theorem.
\end{proof}
Hence $m_k(t),\ t\geq0$ is nonnegative continuous square integrable martingale. It implies that if $m_k$ hits zero, then it equals zero at any moment of time after the moment of hitting zero. Due to the to the time-change theorem \cite{22} (p. 352, Th. 18.4) there exists one-dimensional Brownian motion $B(t),\ t\geq0$ such that
\begin{equation}
\label{eq5.24}
m_k(t)=m_k(0)+B(\langle m_k\rangle(t)),\ t\geq0.
\end{equation}
Denote by
$$
\tau^B_{-b}=\inf\{t\geq0:\ b+B(t)=0\}.
$$
Then $m_k$ admits representation
$$
m_k(t)=b+B(\langle m_k\rangle(t)\wedge\tau^B_{-b}).
$$
The function $b+B(\cdot\wedge\tau^B_{-b})$ is continuous on $[0;+\infty).$
Since there exists the limit of quadratic characteristics $\langle m_k\rangle(\infty),$ then there exists
$$
m_k(\infty)=b+B(\langle m_k\rangle(\infty)\wedge\tau^B_{-b}).
$$
Using the boundness of function $b+B(\cdot\wedge\tau^B_{-b})$ one can conclude that $m_k(\infty)\in[0;+\infty).$
Therefore we proved the following statement
\begin{thm}
\label{thm5.2}
\begin{equation}
\label{eq5.16}
\lim_{t\to+\infty}\exp\Big\{a(k-1)t-\frac{c(k-1)t}{2}-\frac{c(k-1)^2t}{2}\Big\}\cT^{x(\eta,t)}_k\in[0;+\infty),\ a.s.
\end{equation}
\end{thm}
Note that $m_k(\infty)=0$ if $\langle m_k\rangle(\infty)=\tau^B_{-b}.$ Moreover if $\langle m_k\rangle(\infty)>\tau^B_{-b},$ then
$$
E\langle m_k\rangle(\infty)>E\tau^B_{-b}=+\infty.
$$
Then to guarantee that the limit in \eqref{eq5.16} does not equal to zero with positive probability one can try to check that
\begin{equation}
\label{eq5.25}
E\langle m_k\rangle(+\infty)<+\infty.
\end{equation}
Relation \eqref{eq5.25} will be examined in terms of asymptotics of
$E\langle m_k\rangle(t)$ as $t\to+\infty.$ To do this we need
\begin{lem}
\label{lem5.9}
There exist constants $c_1,c_2>0$ such that for all $u_1,u_2\in\mbR^2$ and $t\geq0$ which satisfy the relation
$$
e^{-at}<\|u_2-u_1\|<1
$$
the following estimate holds
$$
P\{\|x(u_2,t)-x(u_1,t)\|\leq1\}\leq\Big(c_1\ln\frac{1}{\|u_2-u_1\|}+c_2\Big)\frac{1}{at}.
$$
\end{lem}
\begin{proof}
It follow from \eqref{eq5.10} that for any $u_1,u_2\in\mbR^2$
$$
d(x(u_2,t)-x(u_1,t))=a(x(u_2,t)-x(u_1,t))dt+d\tilde{\beta}(t),
$$
where
$$
\tilde{\beta}(t)=\int^t_0\int_{\mbR^2}(\varphi(x(u_2,s)-p)-\varphi(x(u_2,s)-p))W(dp,ds),\ t\geq 0
$$
is a martingale with respect to $\cF_t.$ Put $y(t)=x(u_2,t)-x(u_2,t).$ Then one obtain the following Cauchy problem
\begin{equation}
\label{eq5.26}
\begin{cases}
dy(t)=ay(t)dt+d\tilde{\beta}(t)\\
y(0)=u_2-u_1.
\end{cases}
\end{equation}
It is not difficult to check that solution to \eqref{eq5.26} has the representation
$$
y(t)=(u_2-u_1)e^{at}+\int^t_0e^{a(t-s)}d\tilde{\beta}(s).
$$
Note that $y(t)e^{-at},\ t\geq0$ is square integrable martingale with quadratic  characteristic of coordinates
$$
\langle y_ie^{-a\cdot}\rangle(t)=\int^t_0e^{-2as}d\langle\tilde{\beta}\rangle(s),\ i=1,2
$$
and joint characteristic
$$
\langle y_1,y_2\rangle(t)=0.
$$
Here
$$
\langle\tilde{\beta}\rangle(t)=\int^t_0\int_{\mbR^2}((\varphi(x(u_2,s)-p)-\varphi(x(u_2,s)-p))^2dpds.
$$
Consequently $y(t)e^{-at},\ t\geq0$ is an isotropic martingale and by time-change theorem \cite{22}(p. 352, Th. 18.4) there exists a planar Brownian motion $B$ such that
$$
y(t)e^{at}=u_2-u_1+B(\langle y_1e^{-a\cdot}\rangle(t)).
$$
Then
$$
P\{y(t)\in B(0;1)\}=
$$
$$
=P\{y(t)e^{-at}\in B(0;e^{-at})\}=
$$
$$
=P\{B(\langle y_1e^{a\cdot}\rangle(t))\in B(u_2-u_1,e^{-at})\}\leq
$$
$$
\leq P\{\min_{[0;\langle y_1e^{a\cdot}\rangle(t)]}\|u_2-u_1-B(s)\|\leq e^{-at}\}.
$$
To finish the proof we need
\begin{lem}
\label{lem5.10}\cite{30}
Let $B(t),\ t\geq0$ be planar Brownian motion. For $0<r<b<1$  there exists positive constants $c_1$ such that for
$$
P\Big(\min_{[0;1]}\|(b,0)+B(t)\|\leq r\Big)<\Big[c_1+\ln\frac{1}{b}\Big]\ \Big(\ln\frac{1}{r}\Big)^{-1}.
$$
\end{lem}
Since distribution of planar Brownian is rotation-invariant, then
$$
P\{\min_{[0;\langle y_1e^{a\cdot}\rangle(t)]}\|u_2-u_1-B(s)\|\leq e^{-at}\}=
$$
$$
=P\{\min_{[0;\langle y_1e^{a\cdot}\rangle(t)]}\|(\|u_2-u_1\|,0)+B(s)\|\leq e^{-at}\}.
$$
Note that
$$
\langle\tilde{\beta}\rangle(t)=\int^t_0\int_{\mbR^2}(\varphi(x(u_2,s)-p)-\varphi(x(u_2,s)-p))^2dpds=
$$
$$
=2t-2\int^t_0\varphi\ast\varphi(x(u_2,s)-x(u_1,s))ds\leq 2t.
$$
Hence
$$
\langle y_1 e^{-a\cdot}\rangle(t)\leq 2\int^{\infty}_0e^{-2as}ds=\frac{1}{a}.
$$
Consequently
$$
=P\{\min_{[0;\langle y_1e^{a\cdot}\rangle(t)]}\|(\|u_2-u_1\|,0)+B(s)\|\leq e^{-at}\}\leq
$$
$$
\leq P\{\min_{[0;\frac{1}{a}]}\|(\|u_2-u_1\|,0)+B(s)\|\leq e^{-at}\}=
$$
\begin{equation}
\label{eq5.27}
P\{\min_{[0;1]}\|(\sqrt{a}\|u_2-u_1\|,0)+B(s)\|\leq \sqrt{a}e^{-at}\}.
\end{equation}
It follows from Lemma \ref{lem5.10} that \eqref{eq5.27} less or equal to
$$
\Big[c_1+\ln_{+}\frac{1}{\sqrt{a}\|u_2-u_1\|}\Big]\frac{1}{\ln\frac{1}{\sqrt{a}}+at}
$$
which completes the proof of lemma.
\end{proof}
Lemma \ref{lem5.9} allows to obtain the following estimate for expectation of quadratic characteristic of the martingale $m_k(t),\ t\geq0.$
\begin{thm}
\label{thm5.2} There exist positive constants $k_1,k_2$ such that for any $t\geq0$
$$
E\langle m_k\rangle(t)\leq e^{3(k-1)^2ct}\ \frac{t\ (k_1+\frac{k_2}{\sqrt{a}})^{\frac{1}{2}}}{(\sqrt{a}+\ln 2+at)^{\frac{1}{2}}+(\sqrt{a}+\ln 2)^{\frac{1}{2}}}.
$$
\end{thm}
\begin{proof}
Since
$$
E\langle m_k\rangle(t)=(k-1)^2E\int_{\mbR^2}\int_{\mbR^2}\int^t_0\cE(u,s)\cE(v,s)\cdot
$$
$$
\cdot\nabla \varphi^1\ast\nabla\varphi^1(x(u,s)-x(v,s))ds\ \nu_k(du)\nu_k(dv)+
$$
$$
+(k-1)^2E\int_{\mbR^2}\int_{\mbR^2}\int^t_0\cE(u,s)\cE(v,s)\cdot
$$
\begin{equation}
\label{eq5.28}
\cdot\nabla \varphi^2\ast\nabla\varphi^2(x(u,s)-x(v,s))ds\ \nu_k(du)\nu_k(dv),
\end{equation}
then it suffices to estimate one summand in \eqref{eq5.28}. Note
$$
E\int_{\mbR^2}\int_{\mbR^2}\int^t_0\cE(u,s)\cE(v,s)\cdot
$$
$$
\cdot\nabla \varphi^1\ast\nabla\varphi^1(x(u,s)-x(v,s))ds\ \nu_k(du)\nu_k(dv)=
$$
$$
=E\int_{\mbR^2}\int_{\mbR^2}\int^t_0\cE(u,s)\cE(v,s)\cdot
$$
$$
\cdot\nabla \varphi^1\ast\nabla\varphi^1(x(u,s)-x(v,s))1_{\{\|x(u,s)-x(v,s)\|\leq 2\}}ds\ \nu_k(du)\nu_k(dv)\leq
$$
$$
\leq c_1\ E\ \Big(\int_{\mbR^2}\int_{\mbR^2}\int^t_0\cE(u,s)\cE(v,s)\cdot
$$
$$
\cdot1_{\{\|x(u,s)-x(v,s)\|\leq 2\}}ds\ \nu_k(du)\nu_k(dv)|\cF_0\Big)=
$$
$$
= c_1\ E^\eta\int_{\mbR^2}\int_{\mbR^2}\int^t_0E^W\cE(u,s)\cE(v,s)\cdot
$$
\begin{equation}
\label{eq5.29}
\cdot1_{\{\|x(u,s)-x(v,s)\|\leq 2\}}ds\ \nu_k(du)\nu_k(dv),\ c_1>0.
\end{equation}
It follows from Cauchy inequality that
$$
E^W\cE(u,s)\cE(v,s)1_{\{\|x(u,s)-x(v,s)\|\leq 2\}}\leq
$$
$$
\leq \Big(E^W\cE(u,s)^4 E^W\cE(u,s)^4\Big)^{\frac{1}{4}}\Big(P^W\Big\{\|x(u,s)-x(v,s)\|\leq 2\Big\}\Big)^{\frac{1}{2}}.
$$
It can be checked that
$$
E^W\cE(u,s)^4=e^{6(k-1)^2 ct}.
$$
Then
$$
\Big(E^W\cE(u,s)^4 E^W\cE(u,s)^4\Big)^{\frac{1}{4}}=e^{3(k-1)^2 ct}.
$$
Applying Lemma \ref{lem5.9} one can conclude
$$
P^W\{\|x(u,s)-x(v,s)\|\leq 2\}\leq
$$
$$
\leq\Big[c_2+\ln_{+}\frac{1}{\sqrt{a}\|u-v\|}\Big]\frac{1}{\ln\frac{1}{\sqrt{a}}+\ln 2+as}.
$$
Hence \eqref{eq5.29} less or equal to
$$
c_1\ E^\eta\int_{\mbR^2}\int_{\mbR^2}\int^t_0 e^{3(k-1)^2cs}\cdot
$$
\begin{equation}
\label{eq5.30}
\cdot
\Big(\Big[c_2+\ln_{+}\frac{1}{\sqrt{a}\|u-v\|}\Big]\frac{1}{\ln\frac{1}{\sqrt{a}}+\ln 2+as}\Big)^{\frac{1}{2}} ds\nu_k(du)\nu_k(dv).
\end{equation}
Note that
$$
\int^t_0e^{3(k-1)^2cs}\frac{1}{\sqrt{\ln\frac{1}{\sqrt{a}}+\ln 2+as}}ds\leq
$$
$$
\leq e^{3(k-1)^2ct}\ \frac{2}{a}\Big(\sqrt{\sqrt{a}+\ln 2+at}-\sqrt{\sqrt{a}+\ln 2}\Big).
$$
Consequently \eqref{eq5.30} less or equal to
$$
c_1\Big(e^{3(k-1)^2ct}\ \frac{2}{a}\Big(\sqrt{\sqrt{a}+\ln 2+at}-\sqrt{\sqrt{a}+\ln 2}\Big)\Big)\cdot
$$
\begin{equation}
\label{eq5.31}
\cdot E^\eta\int_{\mbR^2}\int_{\mbR^2}
\Big(c_2+\ln_{+}\frac{1}{\sqrt{a}\|u-v\|}\Big)^{\frac{1}{2}}\nu_k(du)\nu_k(dv)
\end{equation}
Applying H$\ddot{o}$lder inequality one can obtain that \eqref{eq5.31} less or equal to
$$
\Big(E\nu_k(\mbR^2)^2\Big)^\frac{1}{2}\cdot
$$
\begin{equation}
\label{eq5.33}
\cdot \Big(\Big(c_2+\frac{1}{\sqrt{a}}\Big)E\nu_k(\mbR^2)^2+E\int_{\mbR^2}\int_{\mbR^2}\ln_+\frac{1}{\|u-v\|}\nu_k(du)\nu_k(dv)\Big)^{\frac{1}{2}}.
\end{equation}
It was proved in Lemma \ref{lem2.8} that
$$
E\int_{\mbR^2}\int_{\mbR^2}\ln_+\frac{1}{\|u-v\|}\nu_k(du)\nu_k(dv)<+\infty.
$$
Set
$$
c_3=\Big(E\nu_k(\mbR^2)^2\Big)^\frac{1}{2}.
$$
Then \eqref{eq5.33} less or equal to
$$
c_3\Big((c_2+\frac{1}{\sqrt{a}})c^2_3+c_4\Big)^\frac{1}{2},\ c_4>0.
$$
\end{proof}
\section*{Acknowledgments}
The third named author Olga Izyumtseva thanks Sasha Sodin for useful discussions. Olga Izyumtseva was supported by the European Research Council starting grant 639305 (SPECTRUM).


\begin{thebibliography}{99}
\bibitem{16}
A.A. Dorogovtsev Measure-valued processes and stochastic flows // Proceedings of Institute of Mathematics of NAS of Ukraine. Mathematics and its Applications 66, NASU, Institute of Mathematics, Kiev, 2007, 290 p.
\bibitem{9}
A.A. Dorogovtsev, Srochastic flows with interaction and measure-valued processes, International Journal of Mathematics and Mathematical Sciences 63 (2003), 3963-3977.
\bibitem{31}
Le Gall J.-F., Wiener sausage and self-intersection local times, Journal of Functional Analysis 88 (1990), 299-341.
\bibitem{10}
T.Funaki, Random motion of strings and related stochastic evolution equations, Nagoya Math. J. 89 (1983), 129-193.
\bibitem{15}
Da Parato, J. Zabchyk, Stochastic equations in infinite dimensions, Encyclopedia of Mathematics and its Application 45, Cambridge Univ. Press, Cambridge, 1992, 454 p.
\bibitem{11}
M.Kardar, G.Parisi, Y.-C.Zhang, \textit{Dynamic Scaling of Growing Interfaces}, Physical Reviev Letters 56 (1986), no. 9, 889-892.
\bibitem{13}
V.I. Arnold, B.A. Khesin, Topological Methods in Hydrodynamics, Springer-Verlag, New York, 1998, 392 p.
\bibitem{12}
M. Cranston, Y. LeJan, Geometric evolution under isotropic stochastic flow, Electron J. Probab. 3, paper 4 (1998), 1-36.
\bibitem{14}
C.L. Zirbel, E. Cinlar, Dispersion of particle systems in Brownian flows, Adv. Appl. Probab. 28 (1996), 53-74.
\bibitem{1}
S. Varadhan, Appendix to: Euclidian quantum field theory, by K. Symanzik,R. Jost, New York, 1969.
\bibitem{2}
J. Rosen, A renormalized local time for multiple intersections of planar Brownian motion, Sem.de Prob. XX, 20 (1986), 515-531.
\bibitem{3}
E.B.Dynkin, Regularized self-intersection local times of planar Brownian motion, The Annals of Probability, 16 (1988), no. 1, 58-74.
\bibitem{27}
A.A. Dorogovtsev, O.L. Izyumtseva, Hilbert-valued self-intersection local times for planar Brownian motion, Stochastics,
91 (2019), no.1, 143-154.
\bibitem{25}
A.A. Dorogovtsev, O.L. Izyumtseva, On regularization of the formal Fourier-Wiener transform of the self-intesection local time of a planar Gaussian process, Theory of Stochastic Processes, 17 (2011), no. 1, 28-38.
\bibitem{26}
A.A. Dorogovtsev, O.L. Izyumtseva, Local time of self-intersection, Ukrainian Mathematical Journal, 68 (2016), no. 3, 325-379.
\bibitem{7}
J. Cuzick, J.P. DuPreez, Joint continuity of Gaussian local times, The Annals of Probability, 10 (1982), no. 3, 810-817.

\bibitem{5}
A.V. Rudenko, Local time as an element of the Sobolev space, Theory Stoch. Proc., 13, 3 (2007), no. 29, 65-79.
\bibitem{18}
H. Kunita, Stochastic flows and stochastic differential equations, Cambridge University Press, 1997.
\bibitem{22}
O. Kallenberg, Foundation of Modern Probability, Springer, 1997.
\bibitem{21}
P. Kotelenez, Stochastic ordinary and stochastic partial differential equations, Springer, 2008.
\bibitem{23}
R. Liptser, A.N. Shiryayev, Theory of Martingales, Mathematics and its Application, 1989.

\bibitem{30}
Dvoretzky A., Erdos P., Kakutani S., Multiple points of paths of Brownian mnotion in the plane, Bulletin of the research counsil of Israel (1954), 364-371.
\end{thebibliography}
\end{document}